\documentclass[]{amsart}
\usepackage{textcomp, amsmath, amssymb, amsthm, paralist, pdfsync, ulem, nicefrac, url, bbm, microtype}

\usepackage[usenames,dvipsnames]{color}

\DeclareFontFamily{OT1}{pzc}{}
\DeclareFontShape{OT1}{pzc}{m}{it}{<-> s * [1.10] pzcmi7t}{}
\DeclareMathAlphabet{\mathpzc}{OT1}{pzc}{m}{it} 
\DeclareMathOperator{\tr}{tr}
\DeclareMathOperator{\Dix}{\mathrm{Tr}_{\mathpzc{W}}}
\DeclareMathOperator{\LimOmega}{\mathrm{Lim}_{\mathpzc{W}}}
\DeclareMathOperator{\Var}{\mathrm{Var}}

\newtheorem{thm}{\bf \textsc{Theorem}}[section]
\newtheorem{cor}[thm]{\bf Corollary}
\newtheorem{lem}[thm]{\bf Lemma}

\theoremstyle{definition}

\theoremstyle{definition}
\newtheorem*{rmk}{\it Remark}

\newtheorem{defn}[thm]{\bf Definition}{\theoremstyle{definition}
\newtheorem*{defnnonum}{\bf Definition}}

\renewcommand{\emptyset}{\varnothing}
\renewcommand{\chi}{\mathbbm{1}}
\renewcommand{\epsilon}{\varepsilon}

\newcommand{\Nint}{\mbox{\sout{$\displaystyle{\int}$}} \;}
\newcommand{\nint}{\mbox{\sout{${\int}$} \;}}

\renewcommand{\leq}{\leqslant}
\renewcommand{\geq}{\geqslant}
\begin{document}
\title{Spectral metric spaces for Gibbs measures}
\author{M.~Kesseb\"ohmer}
\address{Fachbereich Mathematik, Universit\"at Bremen, 28359 Bremen, Germany}
\email{mhk@math.uni-bremen.de}
\author{T.~Samuel}
\address{School of Mathematics and Statistics, University of St Andrews, St Andrews, KY169SS, Scotland}
\email{as57@st-andrews.ac.uk}
\begin{abstract}
We construct  spectral metric spaces for Gibbs measures on a one-sided topologically exact subshift of finite type.  That is, for a given Gibbs measure we construct a spectral triple and show that  Connes' corresponding pseudo-metric is a metric and that its metric topology agrees with the weak-${*}$-topology on the state space over the set of continuous functions defined on the subshift.   Moreover, we show that each Gibbs measure   can be fully recovered from the noncommutative integration theory  and that the noncommutative volume constant of the associated spectral triple is equal to the reciprocal of the measure theoretical entropy of the shift invariant Gibbs measure. 
\end{abstract}


\maketitle

\section{Introduction}

In this paper we contribute to the on going research on formulating a noncommutative fractal geometry.  The starting point of our investigations is the work of Antonescu-Ivan and Christensen \cite{STCS} where the metric aspects of a spectral triple for an AF (approximately finite) $C^{*}$-algebra are considered.

One aspect of noncommutative geometry, or more precisely, the notion of a spectral triple, is to analyse geometric spaces using operator algebras, particularly $C^{*}$-algebras.  This idea first appeared in the work of Gelfand and Na\u{\i}mark \cite{Gelfand+Neumark}, where it was shown that a $C^{*}$-algebra can be seen as a generalisation of the ring of complex-valued continuous functions on a locally compact metric space.  In the 1980s Connes \cite{NCDGCones, C} formalised the notion of noncommutative geometry and, in doing so, showed that the tools of Riemannian geometry can be extended to certain non-Hausdorff spaces known as ``bad quotients'' and to spaces of a ``fractal'' nature.  In particular, Connes proposed the concept of a spectral triple.

\begin{defnnonum}
A \textit{spectral triple} is a triple $(A,H,D)$ consisting of a $C^{*}$-algebra $A$, which acts faithfully on a separable Hilbert space $H$, and an essentially self-adjoint unbounded operator $D$ defined on $H$ with compact resolvent such that the set
\[
\{ a \in A \, : \, \text{the operator} \; [D, \pi(a)] \; \text{extends to a bounded operator defined on} \; H \}
\]
is dense in $A$.  (Here $\pi: A \to B(H)$ denotes the faithful action of $A$ on $H$.) 
\end{defnnonum}

Connes showed that with such a structure one can obtain a pseudo-metric on the state space $\mathcal{S}(A)$ of $A$, analogous to how the Monge-Kantorovitch metric is defined on the space of probability measures on a compact metric space.  In 1998 Rieffel \cite{Ri2} and Pavlovi\'c \cite{Pav}, independently, established conditions under which Connes' pseudo-metric is a metric and established conditions under which the metric topology of Connes' pseudo-metric is equivalent to the weak-${*}$-topology defined on $\mathcal{S}(A)$ - Theorem \ref{RieffelandPavlovicThmNCG}.  In the situation that Connes' pseudo-metric is a metric we call the spectral triple a \textit{spectral metric space}.  Further, Connes noted that the Dixmier trace provides the proper analogue of integration in the contexts of a spectral triple and  developed a notion of dimension, called metric dimension.  To justify that his approach was the correct approach, in \cite{NCDGCones, C} he proved that for an arbitrary smooth compact spin$^{c}$ Riemannian manifold there exists a spectral triple from which the metrical information, the measure theoretical information, the smooth structure of the manifold and much else can be recovered.  This illustrates that a spectral triple allows one to move beyond the limits of classical Riemannian geometry.  That is to say, not only is one able to recover classical aspects of Riemannian geometry, but through the notion of a spectral triple one is able to extend the tools of Riemannian geometry to situations that present themselves at the boundary of classically defined objects, for instance, objects which ``live'' on the boundary of Teichm\"uller space (such as the noncommutative torus, see also \cite{homologyatinfinity}) or those of a ``fractal'' nature (such as the middle third Cantor set).  Although one of the original motivations for noncommutative geometry was to be able to deal with non-Hausdorff spaces, such as foliated manifolds, which are often best represented by a noncommutative $C^{*}$-algebra (see \cite{C, Mar, Ri3, IntroNC,KSS}), this new theory has scope, even when the $C^{*}$-algebra is commutative.

\subsection{A brief history of a fractal noncommutative geometry}

In Connes' seminal book \cite{C}, the concept of a noncommutative fractal geometry is introduced.  Consequently, a remarkable amount of interest has developed in this subject.  In Chapter $\mathrm{IV}$ of \cite{C}, numerous examples are given to indicate how fractal sets can be represented by spectral triples.  Connes' examples include non-empty compact totally disconnected subsets of $\mathbb{R}$ with no isolated points, Julia sets of endomorphisms of the complex plane and limit sets of Fuchsian groups of the second kind.  Subsequently, in $1997$ Lapidus \cite{Lapidus3} proposed several ways in which the notions of a noncommutative fractal geometry could be extended, after which several important articles on the subject appeared.  For instance, Guido and Isola \cite{GI1} analysed the spectral triple presented by Connes for limit fractals in $\mathbb{R}$ which satisfy a certain separation condition.  (Note that such sets are non-empty compact totally disconnected and have no isolated points.)  There, the authors investigated aspects of Connes' pseudo-metric, the metric dimension and the noncommutative integral of Connes' spectral triple.  In \cite{GI2} this construction and analysis is extended to limit fractals in $\mathbb{R}^{n}$, for $n \in \mathbb{N}$.  In \cite{Self-reference-1} Falconer and Samuel have modified this construction to describe multifractal phenomena.  Further, Antonescu-Ivan and Christensen  \cite{STCS} have provided a construction of a spectral triple for an AF $C^{*}$-algebra with particular focus on aspects of Connes' pseudo-metric.   In \cite{DOST} the authors give several examples of spectral triples which represent fractal sets such as the von Koch curve and the Sierpi\'nski gasket.  There, the authors showed that for such sets the Hausdorff dimension can be recovered and that Connes' pseudo-metric induces a metric equivalent to the metric induced by the ambient space on the given set.  More recently, in \cite{Bellisard} the authors adapt Connes' spectral triple to represent the code space $\{ 0, 1 \}^{\mathbb{N}}$ equipped with an ultra-metric $d$.   There it is shown that an adaptation of Connes' pseudo-metric gives rise to a metric equal to $d$.  Further, they proved that the box-counting dimension can be recovered and that a noncommutative integration theory gives rise to an integral with respect to the normalised $\delta$-dimensional Hausdorff measure on the metric space $(\{ 0, 1 \}^{\mathbb{N}}, d)$.  (Here, $\delta$ denotes the Hausdorff dimension of $(\{ 0, 1\}^{\mathbb{N}}, d)$.)

\subsection{Main results}

We generalise the notion of a Haar basis which is well established for the the middle third Cantor set to the setting of a one-sided topologically exact subshift of finite type $(\Sigma_{\mathpzc{A}}, \sigma)$ equipped with a Gibbs measure $\mu_{\phi}$ for a H\"older continuous potential function $\phi \in C(\Sigma_{\mathpzc{A}}; \mathbb{R})$ -- Theorem \ref{Lemma-Generlised-Haar-Basis}.  Using such a construction, in Theorem \ref{Subshiftoffinitetypethmspectraltripleandother} we refine the methods of Antonescu-Ivan and Christensen \cite{STCS} and derive a spectral triple for each such measure space.  From such a representation we shown that a variety of geometric and measure theoretic information can be recovered.  In particular, Theorem \ref{thm-connes-metric-adapt-dirac-AC} shows that Connes' pseudo-metric, given by our spectral triple, is a metric on the state space $\mathcal{S}(C(\Sigma_{\mathpzc{A}}; \mathbb{C}))$ of the $C^{*}$-algebra $C(\Sigma_{\mathpzc{A}}; \mathbb{C})$, and that the topology induced by this metric is equivalent to the weak-${*}$-topology defined on $\mathcal{S}(C(\Sigma_{\mathpzc{A}}; \mathbb{C}))$.  Hence, the spectral triple we provide is a spectral metric space.  In Theorem \ref{Subshiftoffinitetypethmspectraltripleandotherintundmetricdim} it is shown that the metric dimension of our spectral triple is equal to one and that the noncommutative integration theory of our spectral triple is capable of recovering the measure $\nu_{\phi}$.  Moreover, it is established that the noncommutative volume constant is equal to the reciprocal of the measure theoretical entropy of $\nu_{\phi}$ with respect to the left shift $\sigma$.

\subsection{Organization of paper}

Our work is organised as follows.  In Section \ref{SpectralTripesIntro} we define the notion of an unbounded Fredholm module and the notion of a spectral triples.  We then introduce Connes' pseudo metric, the metric dimension and Connes' noncommutative integration theory.  In Section \ref{DynamicalSysIntro}, a review the theory of the theromodynamic formalism on shift spaces as developed by Bowen and Ruelle \cite{Bowen2, Bowen1, Ruelle1} is presented.  It is here that we include the description, of the generalised notion, of a Haar basis for a one-sided topologically exact subshift of finite type equipped with a Gibbs measure.  We also introduce the relevant notions of renewal theory and deduce two useful counting results for topologically exact subshifts of finite type - Corollary \ref{Cor1} and Corollary \ref{Cor2}.  In Section \ref{MainSection}, we construct a spectral triple for a topologically exact subshift of finite type equipped with a Gibbs measure for a H\"older continuous potential function and present the geometric and measure theoretical results (Theorems \ref{Subshiftoffinitetypethmspectraltripleandotherintundmetricdim} and \ref{thm-connes-metric-adapt-dirac-AC}) as described above.

\section{Unbounded Fredholm modules and spectral triples}\label{SpectralTripesIntro}

In \cite{NCDGCones, CO2, C} Connes extends the notion of a compact metric space to the setting of $C^{*}$-algebras and unbounded operators on Hilbert spaces.  This is done in terms of unbounded Fredholm modules and spectral triples.

\begin{defn}\label{STdefn}
Let $A$ denote a unital $C^{*}$-algebra.  An \textit{unbounded Fredholm module} $(H, D)$ over $A$, consists of a separable Hilbert $H$, a  faithful $*$-representation $\pi: A \to B(H)$ and an operator $D$, called a Dirac operator, such that the following hold.  The operator $D$ is an essentially self-adjoint, unbounded linear operator with compact resolvent, such that the set
\[
\{ a \in A \; : \; [D, \pi(a)] \; \text{is densely defined and extends to a bounded operator on} \; H \}
\]
is norm dense in $A$.  A triple $(A, H, D)$ with the above description is called a \textit{spectral triple}.
\end{defn}

By Gelfand and Na\u{\i}mark's classification theorem for $C^{*}$-algebras \cite{Gelfand+Neumark} it is necessarily the case, that given a $C^{*}$-algebra $A$ there exists a Hilbert space $H$ and a faithful $*$-representation $\pi: A \to B(H)$.  Since such a $*$-representation is not necessarily unique, it is important to specify the $*$-representation.  For this reason we occasionally use the notation $(A, H, \pi, D)$ for a spectral triple.  Also, observe that in Definition \ref{STdefn} it is not necessary for the $C^{*}$-algebra to be unital.

The compact resolvent property of $D$ in Definition \ref{STdefn} can be regarded as a generalisation of the ellipticity property of the standard Dirac operator defined on a compact smooth Riemannian manifold (see \cite{Mar}).  The condition that the closure of $[D, \pi(a)]$ is densely defined and extends to a bounded operator is analogous to a Lipschitz condition (see \cite{Bellisard2, Mar}).

Often Definition \ref{STdefn} is strengthen to include a summability statement.  Let $p > 0$ be given.  A spectral triple $(A, H, D)$ is said to be \textit{$p$-summable} if and only if
\begin{equation}\label{hash1}
\tr\left(\left(1 + D^{2}\right)^{-\frac{p}{2}}\right) < \infty;
\end{equation}
it is called \textit{$(p, +)$-summable}, for $p > 1$,  if and only if
\begin{equation}\label{hash2}
\limsup_{N \to \infty} \frac{1}{N^{1 - 1/p}} \sum_{k = 1}^{N} \sigma_{k}\left( (1 + D^{2})^{1/2} \right) < \infty;
\end{equation}
and it is call \textit{$(1, +)$-summable} if and only if
\begin{equation}\label{hash3}
\limsup_{N \to \infty} \frac{1}{\ln(N)} \sum_{k = 1}^{N} \sigma_{k}\left( (1 + D^{2})^{1/2} \right) < \infty.
\end{equation}
Here, for a Hilbert space $H$, an operator $T \in B(H)$ and $k \in \mathbb{N}$, we let $\sigma_{k}(T)$ denote the $k$-th largest singular value, including multiplicities, of $T$.  A spectral triple is called \textit{finitely summable} if $(\ref{hash1})$ holds for some real $p > 0$.  This is equivalent to there existing a $p > 0$ such that 
\begin{equation}\label{hash2}
\limsup_{N \to \infty} \frac{1}{\ln(N)} \sum_{k = 1}^{N} \sigma_{k}\left(\left(1 + D^{2}\right)^{-p/2}\right) < \infty.
\end{equation}
The proof of this equivalence is given by the following analytic result on sequences.

\begin{lem}\label{lemma-pplus=pimpliestheta}
Let $\{ x_{k} \}_{k \in \mathbb{N}}$ denote a monotonically increasing unbounded sequence of positive real numbers and set
\begin{eqnarray*}
d_{1} &\mathrel{:=}& \sup \left\{ \alpha \geq 0 \, : \, \limsup_{N \to \infty} \frac{\sum_{k = 1}^{N} {x_{k}}^{-\alpha}}{\ln(N)} = \infty \right\},\\
d_{2} &\mathrel{:=}& \inf \left\{ \alpha \geq 0 \, : \, \limsup_{N \to \infty} \frac{\sum_{k = 1}^{N} {x_{k}}^{-\alpha}}{\ln(N)} = 0 \right\},\\
d_{3} &\mathrel{:=}& \inf \left\{ \alpha \geq 0 \, : \, \sum_{k = 1}^{\infty} {x_{k}}^{-\alpha} < \infty \right\} \quad = \quad \sup \left\{ \alpha \geq 0 \, : \, \sum_{k = 1}^{\infty} {x_{k}}^{-\alpha} = \infty \right\},\\
d_{4} &\mathrel{:=}& \left( \liminf_{k \to \infty} \frac{\ln(x_{k})}{\ln(k)} \right)^{-1} \quad \text{(where we use the convention that $1/\infty = 0$)}.
\end{eqnarray*}
If any of the above are positive and finite, then they are all equal.
\end{lem}

\begin{proof}
The proof of this lemma is a simple analytic exercise on sequences.
\end{proof}

\subsection{Connes' Pseudo Metric}

In this section we describe Connes pseudo-metric on the state space of a $C^{*}$-algebra $A$ which is induced by an unbounded Fredholm module $(H, D)$ over $A$.

A compact metric space $X$ naturally embeds into the state space $\mathcal{S}(C(X); \mathbb{C})$ of $C(X; \mathbb{C})$ and by the Riesz Representation Theorem we have that $\mathcal{S}(C(X); \mathbb{C})$ coincides with $M(X)$, the space of Borel probability measures on $X$.  By the Banach-Alaoglu Theorem this space is weak-${*}$-compact.  Moreover, letting $\text{Lip}(X; \mathbb{C})$ denote the subset of $C(X; \mathbb{C})$ consisting of Lipschitz continuous functions and letting $\text{Lip}(f)$ denote the Lipschitz constant of $f \in \text{Lip}(X; \mathbb{C})$, the \label{Monge-Kantorovitch-metric}\textit{Monge-Kantorovitch metric} given, for all $\mu, \nu \in M(X)$, by
\[
d_{MK}(\mu, \nu) \mathrel{:=} \sup \bigg\{ \int_{X} f \, d\mu - \int_{X} f \, d\nu \; : \; f \in \text{Lip}(X; \mathbb{C}) \; \text{with} \; \text{Lip}(f) \leq 1 \bigg\},
\]
defines a metric on $M(X)$ whose topology coincides with the weak-${*}$-topology.

Analogously, given a spectral triple $(A, H, \pi, D)$, Connes defines a pseudo-metric $d: \mathcal{S}(A) \times \mathcal{S}(A) \to \mathbb{R}$ given, for each $p, q \in \mathcal{S}(A)$, by 
\begin{align*}
d(p, q) \; \mathcal{:=} \; \sup \{ \lvert p(a) - q(a) \rvert \; : \;& a \in A \; \text{and} \; [D, \pi(a)] \; \text{is densely defined and}\\
&\text{extends to a bounded operator with norm} \; \leq \; 1 \}.
\end{align*}
The term pseudo-metric is used since $d(p, q)$ is not necessarily finite for all $p, q \in \mathcal{S}(A)$.  Conditions under which Connes' pseudo-metric is a metric are given by Rieffel and Pavlovi{\'c} and are re-stated in the following theorem.

\begin{thm}\label{RieffelandPavlovicThmNCG}\textnormal{(Rieffel \cite{Ri2} and Pavlovi{\'c} \cite{Pav})}
For a spectral triple $(A, H, \pi, D)$ the following hold.
\begin{enumerate}
\item Connes pseudo-metric is a metric if and only if the set 
\begin{align*}
\mathcal{A}_{D} \mathrel{:=} \{ a \in A \; : \; &[D, \pi(a)] \; \text{is densely defined and extends to a}\\
&\text{bounded operator defined on $H$ with norm} \; \leq \; 1 \}.
\end{align*}
has a bounded image in the quotient space $A / \{ z \mathbb{I} \, : \, z \in \mathbb{C} \}$, where $\mathbb{I}$ denotes the unit of $A$.
\item The topology induced by Connes pseudo-metric $d$ coincides with the weak-${*}$-topology on $\mathcal{S}(A)$ if and only if the set $A_{D}$ has a totally bounded image in the quotient space $A / \{ z \mathbb{I} \, : \, z \in \mathbb{C} \}$.
\end{enumerate}
\end{thm}

\begin{defn}
A spectral triple $(A, H, D)$ is called a \textit{spectral metric space} if Connes' pseudo-metric is a metric.
\end{defn}

\subsection{Infinitesimals, Measurability and Dimension}\label{SDVF}

Recall that the Hausdorff dimension of a subset $E$ of $\mathbb{R}^{n}$ is given by
\[ 
\inf \{ s > 0 \, : \, \mathcal{H}^{s}(E) = 0 \},
\]
where $\mathcal{H}^{s}$ denotes the $s$-dimensional Hausdorff measure (see \cite{F2} for more details).  The Hausdorff dimension defines a geometric characteristic of the set $E$, and as we will see, is also encoded in the Dirac operator of a spectral triple.  

In this section, we introduce the expectation of a compact operator which arises within the theory of operator algebras and then present the definition of the metric dimension of a spectral triple introduced by Connes \cite{C}.  Having developed these notions, we are then able to define the noncommutative integral which is defined only for a finitely summable spectral triple.

In the theory of operator algebras, and hence, noncommutative geometry, the notion of an expectation/integral of a positive compact operator $T \in K(H)$, defined on a complex separable Hilbert space $H$, is given by the coefficient of logarithmic divergence of the singular values of $T$.  In particular, the ideal of compact operators $K(H)$ provides the ``infinitesimal'' of noncommutative geometry.  Heuristically, in setting of ``standard'' geometry, an infinitesimal is an ``object'' smaller than any feasible measurement and not zero in ``size'', but so small that it cannot be distinguished from zero by any available means.  As a matter of interest, we remark that the founders of calculus, Euler, Leibniz, and Newton initially formulated the theory of calculus using infinitesimals.  However, the notion and definition was foreshadowed in Archimedes' script \textit{The Method of Mechanical Theorems}.

Early attempts to define an expectation within the theory of operator algebras (see \cite{segal}) used ordinary traces of Hilbert space operators, where trace-class operators were thought to be the analogy of integrable functions.  However, it soon became apparent that this is not sufficient.  In 1966 Dixmier \cite{Dixmier} found other tracial states that are more suitable.  He noted that to appropriately define an expectation within the theory of operator algebras, one must suppress infinitesimals of order higher than one.  In particular, one wants to find the coefficient of the divergence rate of the singular values of an infinitesimal operator of order one.

\begin{defn}
Let $H$ denote a complex separable Hilbert space and let $K(H)$ denote the ideal of compact operators in $B(H)$.  Then $T \in K(H)$ is an \textit{infinitesimal of order $s > 0$}, if there exist positive constants $c_{1}, c_{2}$ such that, for each $k \in \mathbb{N}$, we have that
\[
c_{1} k^{-s} \; \leq \; \sigma_{k}(T) \; \leq \; c_{2} k^{-s}.
\]
\end{defn}

\begin{defn}
The \textit{Dixmier ideal} of a separable Hilbert space $H$ is denoted by $\mathcal{L}^{1, +}(H)$ and is defined by
\[
\mathcal{L}^{1, +}(H) \, \mathrel{:=} \, \left\{ T \in K(H) \, : \, \limsup_{N \to \infty} \, \frac{\sum_{k = 1}^{N} \sigma_{k}(T)}{\ln(N)} \, < \, \infty \right\}.
\]
For a state $\mathpzc{W}$ on $l^{\infty}(\mathbb{R}_{+}^{*})$ satisfying the conditions of \cite[Theorem 1.5]{Rennie2}, we define the $\mathpzc{W}$-\textit{Dixmier trace} of a positive linear operator $T \in \mathcal{L}^{1, +}(H)$ by
\begin{equation}\label{measurable-operator-Dix-mier}
\Dix (T) \, \mathrel{:=} \LimOmega \left( \frac{\sum_{k = 1}^{N} \sigma_{k}(T)}{\ln (N)} \right)_{N \in \mathbb{N}},
\end{equation}
whereby, following convention, for $(x_{1}, x_{2}, \dots) \in l^{\infty}(\mathbb{R}_{+}^{*})$, we set
\[
\LimOmega (x_{1}, x_{2}, \dots) \; \mathrel{:=} \; \mathpzc{W}(x_{1}, x_{2}, \dots).
\]
For a general operator in $\mathcal{L}^{1, +}(H)$ the $\mathpzc{W}$-Dixmier trace is defined to be the natural complex linear extension of $\Dix$.  (Here $\mathbb{R}_{+}^{\infty}$ denotes the group of positive
real numbers under multiplication.)
\end{defn}

\begin{defn}
Let $H$ denote a complex Hilbert space and let $I$ denote an ideal of $B(H)$.  A \textit{singular trace on $I$} is a linear functional $\mathpzc{T}$ with domain $I$ such that the following hold.
\begin{enumerate}
\item $\mathpzc{T}$ vanishes on operators with finite dimensional range.
\item If $T_{1}, T_{2} \in I$ are such that $\lim_{k \to \infty} \sigma_{k}(T_{1}) / \sigma_{k}(T_{2}) = 1$, then $\mathpzc{T}(T_{1}) = \mathpzc{T}(T_{2})$.
\item If $T_{1}, T_{2} \in I$ have the property that $\sigma_{k}(T_{1}) \leq \sigma_{k}(T_{2})$ for all but a finite number of $k \in \mathbb{N}$, then $\mathpzc{T}(T_{1}) \leq \mathpzc{T}(T_{2})$.
\item  For all $T_{1}, T_{2} \in I$, we have that $\mathpzc{T}(T_{1}T_{2}) = \mathpzc{T}(T_{2}T_{1})$.
\end{enumerate}
\end{defn}

\begin{thm}\label{DixmierTraceIdealThmIntextVersion}\textnormal{(Carey-Philips-Sukochev \cite{Rennie2})}
Let $H$ denote a complex separable Hilbert space and let $\mathpzc{W}$ denote a state on $l^{\infty}(\mathbb{R}_{+}^{*})$ satisfying the conditions of \cite[Theorem 1.5]{Rennie2}. Then the Dixmier ideal $\mathcal{L}^{1, +}(H)$ is an ideal of $B(H)$ and the functional $\Dix$ is a singular trace.
\end{thm}

\begin{rmk}
The current state of the art on Dixmier traces can be found in \cite{Rennie1, Rennie2}.
\end{rmk}

We now define the notions of a measurable operator and the operator algebra analogue of an expectation.

\begin{defn}
If $T\in \mathcal{L}^{1, +}(H)$ and if $\Dix(T)$ is independent of $\mathpzc{W}$, meaning that the limit
\[
\lim_{N \to \infty} \frac{\sum_{k = 1}^{N} \sigma_{k}(T)}{\ln (N)}
\]
exists, then we call $T$ \textit{measurable}.  The \textit{noncommutative expectation} of a measurable operator $T\in \mathcal{L}^{1, +}(H)$ is  denoted by $\nint T$ and given by
\[
\Nint T \, \mathrel{:=} \, \lim_{N \to \infty} \frac{\sum_{k = 1}^{N} \sigma_{k}(T)}{\ln (N)}.
\]
\end{defn}

The metric dimension of a spectral triple $(A, H, \pi, D)$ is given by the non-negative positive integer $\delta$ to which the singular values of $\lvert 1 + D^{2} \rvert^{-\nicefrac{\delta}{2}}$ form a logarithmically divergent series.  However, such a number does not necessarily have to exist.  In fact, the metric dimension only exists if the spectral triple is finitely summable.

\begin{defn}\label{metricdimdefnst}
Let $(A, H, \pi, D)$ denote a finitely summable spectral triple.  Then the \textit{metric dimension} (sometimes called the \textit{spectral dimension}) of $(A, H, \pi, D)$ is defined to be the non-negative real number
\begin{eqnarray*}\label{delta(A,H,D)-what-else}
\delta \; = \; \delta(A, H, D) &\mathrel{:=}& \, \inf \left\{ p \geq 0 \, : \, \tr((1 + D^{2})^{-\nicefrac{p}{2}}) < \infty \right\}\\
&=& \sup  \left\{ p \geq 0 \, : \, \tr((1 + D^{2})^{-\nicefrac{p}{2}}) = \infty \right\} \\
&=& \sup \left\{ p \geq 0 \, : \, \limsup_{N \to \infty} \frac{1}{\ln(N)} \sum_{k = 1}^{N} \sigma_{k}((1 + D^{2})^{-\nicefrac{p}{2}}) = \infty \right\} \\
&=& \inf \, \left\{ p \geq 0 \, : \, \limsup_{N \to \infty} \frac{1}{\ln(N)} \sum_{k = 1}^{N} \sigma_{k}((1 + D^{2})^{-\nicefrac{p}{2}}) = 0 \right\}\label{delta(A,H,D)-what-else4}.
\end{eqnarray*}
\end{defn}

\begin{rmk}
The dimension of a spectral triple can take the value zero.  The two circumstances under which the dimension is equal to zero are the following.
\begin{enumerate}
\item The singular values of $(1 + D^{2} )^{-\nicefrac{\delta}{2}}$ converge to zero exponentially fast, see \textnormal{\cite{expfast-finte-spec1}} for examples of this case.
\item The algebra and the Hilbert space are finite dimensional, in which case the Dirac operator is a self-adjoint matrix.  Such spectral triples have been fully classified and the classification can be found in \textnormal{\cite{Discret-spectral-triples_and_their_symmetries, Krajewski, Krajewski2}}.
\end{enumerate}
In the case that there does not exist a positive real number $\delta$ such that $( 1 + D^{2})^{-\nicefrac{\delta}{2}}$ is an infinitesimal of order one,  one examines wether there exists a positive number $t$ such that $e^{- t D^{2}}$ is of trace-class.  If this is the case, one says the spectral triple is of infinite dimension (see \cite{CO2, C} for further details on spectral triples with infinite metric dimension).
\end{rmk}

For a $p$-summable spectral triple $(A, H, \pi, D)$, the operator $\lvert D \rvert^{-\delta}$ generalises the notion of a volume form and the Dixmier trace of the operator $\lvert D \rvert^{-\delta}$ generalises the notion of a volume.

\begin{defn}
Let $(A, H, \pi, D)$ denote a finitely summable spectral triple with non-zero metric dimension $\delta$ and let $\mathpzc{W}$ denote a state on $l^{\infty}(\mathbb{R}_{+}^{*})$ satisfying the conditions of \cite[Theorem 1.5]{Rennie2}.  Then the $\mathpzc{W}$-\textit{volume} of $(A, H, \pi, D)$ is defined by
\[
V_{\mathpzc{W}} \, = \, V_{\mathpzc{W}}(A, H, D) \, \mathrel{:=} \, \Dix ( \left\lvert D \right \rvert^{-\delta} ),
\]
where the Dixmier trace is taken over the ideal $\mathcal{L}^{1, +}(\mathrm{ker}(D)^{\perp}) \subseteq B(H)$.  If $\lvert D \rvert^{-\delta}$ is a measurable operator, then the \textit{noncommutative volume constant} of $(A, H, D)$ is defined by
\[
V \, = \, V(A, H, D) \,\mathrel{:=} \, \Nint \left\lvert D \right \rvert^{-\delta},
\]
where the noncommutative integral is taken over the ideal $\mathcal{L}^{1, +}(\mathrm{ker}(D)^{\perp}) \subseteq B(H)$.
\end{defn}

The noncommutative integral given by an unbounded Fredholm module over a unital $C^{*}$-algebra $A$ is defined as follows.

\begin{defn}
Let $(A, H, \pi, D)$ denote a finitely summable spectral triple with non-zero metric dimension $\delta$ and let $\mathpzc{W}$ denote a state on $l^{\infty}(\mathbb{R}_{+}^{*})$ satisfying the conditions of \cite[Theorem 1.5]{Rennie2}.  Then the \textit{$\mathpzc{W}$-noncommutative integral} of an element $a \in A$ with respect to the unbounded Fredholm module $(H, \pi, D)$ is defined to be the complex number
\begin{equation}\label{definition-NC-int}
\Dix(\pi(a) \lvert D \rvert^{-\delta}).
\end{equation}
Here, the Dixmier trace is taken over the ideal $\mathcal{L}^{1, +}(\mathrm{ker}(D)^{\perp}) \subseteq B(H)$.  If $\pi(a)\lvert D \rvert^{-\delta}$ is measurable then we refer to the common values of the Dixmier traces as the \textit{noncommutative integral} of $a$ with respect to the unbounded Fredholm module $(H, \pi, D)$. 
\end{defn}

\begin{rmk}
Let $(A, H, \pi, D)$ denote a finitely summable spectral triple with non-zero metric dimension $\delta$.  If $\lvert D \rvert^{-\delta}$ is a measurable operator, then it is not necessarily the case that $\pi(a) \lvert D \rvert^{-\delta}$ will be a measurable operator, for $a \in A$.
\end{rmk}

\subsection{Connes' prototype example}

The prototype of such a structure is given by the following theorem, and justifies the commonly used statement, \textit{Spin geometries are commutative non-commutative geometries}.

\begin{thm}\label{Connes-Spin-com-non-com}\textnormal{(Connes \cite{NCDGCones, C})}
 Let $M$ denote a compact, smooth, orientable, $2n$-dimensional manifold equipped with a spin$^{c}$ structure, let $A$ denote the $C^{*}$-algebra of continuous complex-valued functions defined on $M$, let $H$ denote the complex Hilbert space generated by the spinor fields of $M$ and let $D$ denote the square root of the  Laplace-Beltrami operator  (given by the Levi-Civita connection).  Letting $\pi: A \to B(H)$ denote the $*$-homomorphism given by pointwise multiplication, we have that $(A, H, \pi, D)$ is a $p$-sum\-mable spectral triple.  Moreover, one recovers from the non-commutative setting the following geometric information.
\begin{enumerate}
\setlength{\itemsep}{0pt}
\item The metric dimension is equal to $2n$.
\item Connes' pseudo-metric induces a metic equivalent to the Riemannian metric on $M$.
\item For each $a \in C^{\infty}(M; \mathbb{C})$, we have that
\[
\Nint \pi(a) \lvert D \rvert^{-2n} =  ((n -1)! 2^{n} \pi^{n} n)^{-1}  \int_{M}  a \; \star(1),
\]
where $\star$ denotes the Hodge star operator.
\end{enumerate}
\end{thm}

\section{Thermodynamic formalism and renewal theorems in symbolic dynamics}\label{DynamicalSysIntro}

Fix $l \in \mathbb{N}$ and let $\mathpzc{A} \mathrel{:=} [a_{j, k}]_{j, k}$ denote an irreducible, aperiodic, $l \times l$ matrix of zero's and one's, called the \textit{adjacency matrix} for the alphabet $\Sigma \mathrel{:=} \{ 1, 2, \dots, l \}$.  Define $\Sigma_{\mathpzc{A}}$ to be the space of all infinite sequences taking values in the alphabet $\{ 1, 2, \dots, l \}$ with transitions allowed by $\mathpzc{A}$, that is,
\[
\Sigma_{\mathpzc{A}} \; \mathrel{:=} \; \left\{ \omega \mathrel{:=} (\omega_{1}, \omega_{2}, \dots) \in \prod_{i = 1}^{\infty} \{1, 2, \dots, l\} \; : \; a_{\omega_{k}, \omega_{k + 1}} = 1 \; \text{for all} \; k \in \mathbb{N} \right\}.
\]
For each $k \in \mathbb{N}$ we define the set of admissible  $k$-words by
\[
\Sigma_{\mathpzc{A}}^{k} \; \mathrel{:=} \; \left\{ x \mathrel{:=} (x_{1}, x_{2}, \dots, x_{k}) \in \prod_{i = 1}^{k} \{1, 2, \dots, l\} \; : \; a_{x_{i}, x_{i+1}} = 1 \right\}
\]
and the set of all admissible finite words by
\[
\Sigma_{\mathpzc{A}}^{*} \; \mathrel{:=} \; \bigcup_{k \in \mathbb{N}} \Sigma_{\mathpzc{A}}^{k}.
\]
Furthermore, for each $k \in \mathbb{N}$ and each $x \mathrel{:=} (x_{1}, x_{2}, \dots, x_{k}) \in \Sigma_{\mathpzc{A}}^{k}$, we define the \textit{cylinder set} of $x$ by
\[
[x] \; \mathrel{:=} \; \left\{ \omega \mathrel{:=} (\omega_{1}, \omega_{2}, \omega_{3}, \dots ) \in \Sigma_{\mathpzc{A}} \; : \; (\omega_{1}, \omega_{2}, \dots, \omega_{k}) = (x_{1}, x_{2}, \dots, x_{k}) \right\}.
\]
The space $\Sigma_{\mathpzc{A}}$ is compact with respect to the topology $\mathcal{T}$ generated by the cylinder sets and this topology is metrizable \cite{ET}. For instance consider the metric $d: \Sigma_{\mathpzc{A}} \times \Sigma_{\mathpzc{A}} \to \mathbb{R}$ given, for $\omega, \upsilon \in \Sigma_{\mathpzc{A}}$, by
\[
d(\omega, \upsilon) \; \mathrel{:=} 2^{- \omega \wedge \upsilon},
\]
where $\omega \wedge \upsilon \mathrel{:=} \max\{\sup\{ n \in \mathbb{N} \, : \, \omega, \upsilon \in [x], \; \text{for some} \; x \in \Sigma_{\mathpzc{A}}^{n} \}, 0 \}$.

The \textit{shift map} $\sigma: \Sigma_{\mathpzc{A}} \to \Sigma_{\mathpzc{A}}$ is given, for each $\omega \mathrel{:=} (\omega_{1}, \omega_{2}, \dots) \in \Sigma_{\mathpzc{A}}$, by
\[
\sigma (\omega) \; \mathrel{:=} \; (\omega_{2}, \omega_{3}, \dots ).
\]
The map $\sigma$ is continuous and surjective and at most $l$-to-one.  We call the dynamical system $(\Sigma_{\mathpzc{A}}, \sigma)$ a \textit{topologically exact subshift of finite type}.  Finally, we define a map $\alpha: \Sigma^{*}_{\mathpzc{A}} \to \mathbb{N}$, for each $k \in \mathbb{N}$ and each $x \mathrel{:=} (x_{1}, x_{2}, \dots, x_{k}) \in \Sigma_{\mathpzc{A}}^{k}$,  by
\[
\alpha(x) \mathrel{:=} \sum_{i \in \{ 1, 2, \dots l \}} a_{x_{k}, i}.
\]
We may interpret $\alpha(x)$ as  the number of ``children'' of  a given $x \in \Sigma_{\mathpzc{A}}^{*}$.

\subsection{The Perron-Frobenius-Ruelle Operator and Equilibrium Measures}

In this section we introduce the notions of a Gibbs measure and an equilibrium measure defined on the space $\Sigma_{\mathpzc{A}}$ as well as the Perron-Frobenius-Ruelle operator for a one-sided, topologically exact subshift of finite type $(\Sigma_{\mathpzc{A}}, \sigma)$.  We will see that the existence of eigenmeasures of the dual of the Perron-Frobenius-Ruelle operator proves the existence of Gibbs measures, where topological pressure appears as the logarithm of the corresponding eigenvalue.   The results of this section were originally presented in the work of Bowen \cite{Bowen2, Bowen1} and Ruelle \cite{Ruelle1}.

To introduce the notion of an equilibrium measure we require the concept of entropy of a measure as introduced by Sina\u{\i} \cite{sinaiBook} and Kolmogorov \cite{Kolmogorov1958} and the notion of a Gibbs measure.  Let $\mathcal{M_{\sigma}} \mathrel{:=} M(\Sigma_{\mathpzc{A}}, \sigma)$ denote the set of $\sigma$-invariant Borel probability measures on $\Sigma_{\mathpzc{A}}$. Then for $\mu \in \mathcal{M_{\sigma}}$ we define the \textit{measure theoretical entropy}  of $\mu$ with respect to $\sigma$ to be the non-zero real number given by
\[
h_{\mu}(\sigma) \, \mathrel{:=} \, \lim_{k \to \infty} \frac{1}{k} \sum_{x \in \Sigma_{\mathpzc{A}}^{k}} - \mu([x]) \ln(\mu[x]).
\]
This limit always exists since the sequence $( \sum_{x \in \Sigma_{\mathpzc{A}}^{k}} - \mu([x]) \ln(\mu[x]))_{k \in \mathbb{N}}$ is subadditive.  For a topologically exact subshift of finite type there always exists a unique measure of \textit{maximal entropy} $\mu$ (also called the \textit{Parry measure}) which maximizes the measure theoretical entropy, that is, $h_{\mu}(\sigma)=\sup_{\nu\in \mathcal{M_{\sigma}}} h_{\nu}(\sigma)$, see for instance \cite{Parry, ET}.   In certain situations, this measure is the combinatorial measure, that is, the measure which weights a cylindrical set $[x]$ with measure $1/\text{card}(\Sigma_{\mathpzc{A}}^{k})$, for each $x \in \Sigma_{\mathpzc{A}}^{k}$.  Such situation include the full-shift-space and symbolic representations of Schottky groups.

In order to introduce the class of Gibbs measures for a one-sided, topologically exact subshift of finite type we define the Birkhoff sums of a function $\phi \in C(\Sigma_{\mathpzc{A}}; \mathbb{R})$.

\begin{defn}
For $\phi \in C(\Sigma_{\mathpzc{A}}; \mathbb{R})$ and $k \in \mathbb{N}_{0}$, let $S_{k}\phi: \Sigma_{\mathpzc{A}} \to \mathbb{R}$ denote the $k$-th \textit{Birkhoff sum} of $\phi$ defined, for $\omega \in \Sigma_{\mathpzc{A}}$, by $S_{0}\phi(\omega)\mathrel{:=}0$ and for $k \geq 1$ by 
\[
S_{k}\phi(\omega) \; \mathrel{:=} \; \sum _{m=0}^{k-1}\, \phi(\sigma^{m}(\omega)).
\]
\end{defn}

\begin{thm}\label{P-FO-}\textnormal{(Bowen \cite{ Bowen1})}
Let $(\Sigma_{\mathpzc{A}}, \sigma)$ denote a one-sided topologically exact subshift of finite type and let $\phi \in C(\Sigma_{\mathpzc{A}}; \mathbb{R})$ denote a H\"older continuous function.  Then there exists a Borel probability measure $\mu_{\phi}$ on $\Sigma_{\mathpzc{A}}$ and a uniquely determined number $P(\phi, \sigma) \in [0, \infty)$ associated to $\phi$, such that for some $c > 1$, we have that
\begin{equation}\label{GibbPropertyMeasureDefn}
c^{-1} \; \leq \; \frac{\mu_{\phi}[(\omega_{1}, \omega_{2}, \dots, \omega_{k})]}{e^{S_{k}\phi(\omega) - kP(\phi, \sigma)}}\; \leq \;c,
\end{equation}
for each $k \in \mathbb{N}$ and for each $\omega \mathrel{:=} (\omega_{1}, \omega_{2}, \dots) \in \Sigma_{\mathpzc{A}}$.  Moreover, to each H\"older continuous potential function $\phi \in C(\Sigma_{\mathpzc{A}}; \mathbb{R})$, there exists a unique $\sigma$-invariant measure satisfying the inequalities given in Equation (\ref{GibbPropertyMeasureDefn}). 
\end{thm}

We refer to a measure satisfying the inequalities given in Equation (\ref{GibbPropertyMeasureDefn}) as a \textit{Gibbs measure} for the \textit{potential} $\phi$.  Further, each such Gibbs measure will have strictly positive entropy and the uniquely determined number $P(\phi, \sigma)$ is called the \textit{topological pressure} of $\phi$, which is also characterised by the variational principle,
\begin{equation}\label{variational-principle}
P(\phi, \sigma) \, = \, \sup \left\{ h_{\mu}(\sigma) + \int_{\Sigma_{\mathpzc{A}}} \phi \, d\mu \; : \; \mu \in \mathcal{M_{\sigma}} \right\}.
\end{equation}
If there exists a measure $\nu \in \mathcal{M_{\sigma}}$, such that
\begin{equation}\label{variationPrinciple}
P(\phi, \sigma) \, = \, h_{\nu}(\sigma) + \int_{\Sigma_{\mathpzc{A}}} \phi \, d\nu,
\end{equation}
then we call $\nu$ an \textit{equilibrium measure} for the potential $\phi$.  Setting $\phi = 0$, it is well-known that there exists a unique equilibrium measure associated to $\phi$, this measure is the measure of maximal entropy.

\begin{defn}
For each $\phi \in C(\Sigma_{\mathpzc{A}}; \mathbb{R})$, define the \textit{Perron-Frobenius-Ruelle operator} $\mathcal{L}_{\phi}: C(\Sigma_{\mathpzc{A}}; \mathbb{R}) \to C(\Sigma_{\mathpzc{A}}; \mathbb{R})$, by
\[
\mathcal{L}_{\phi}(f)(\omega) \; \mathrel{:=} \; \sum_{\upsilon \in \sigma^{-1}(\omega)} e^{\phi(\upsilon)} f(\upsilon)
\]
and denote the dual of the Perron-Frobenius-Ruelle operator by $\mathcal{L}^{*}_{\phi}$.
\end{defn}

By the Riesz Representation Theorem we may identify the positive linear functionals on $C(\Sigma_{\mathpzc{A}}; \mathbb{R})$ with the  finite Borel measures on $(\Sigma_{\mathpzc{A}},\mathcal{B})$ and since $\mathcal{L}_{\phi}$ is positive linear operator on $C(\Sigma_{\mathpzc{A}}; \mathbb{R})$, we identify $\mathcal{L}_{\phi}^{*}$ as an operator on the space of finite Borel measures. Let  $M(\Sigma_{\mathpzc{A}})$ denote the space of  Borel probability measures  on $(\Sigma_{\mathpzc{A}},\mathcal{B})$.

\begin{thm}\label{PFR-E-M-Uniq-thm}\textnormal{(Bowen \cite{ Bowen1} and Ruelle \cite{Ruelle1})}
Let $(\Sigma_{\mathpzc{A}}, \sigma)$ denote a one-sided, topologically exact subshift of finite type and let $\phi \in C(\Sigma_{\mathpzc{A}}; \mathbb{R})$ denote a H\"older continuous function.  Then the following hold.
\begin{enumerate}
\item There exists a unique Borel probability measure $\mu_{\phi} \in M(\Sigma_{\mathpzc{A}})$ such that
\[
\mathcal{L}^{*}_{\phi} \mu_{\phi}\, =\, e^{P(\phi, \sigma)} \mu_{\phi}.
\]
\item The unique measure $\mu_{\phi}$ is a Gibbs measure for the potential $\phi$.
\item If $\psi$ is a H\"older continuous function cohomologous to $\phi$ with respect to $\sigma$, then the associated Borel probability measures, given in part 1, are equal.  (Recall that two continuous functions $\phi, \psi \in C(\Sigma_{\mathpzc{A}}, \mathbb{R})$ are called \textit{cohomologous} if there exists a continuous function $f \in C(\Sigma_{\mathpzc{A}}; \mathbb{R})$ such that $\phi - \psi = f - f \circ \sigma$.)
\item There exists a unique strictly positive eigenfunction $h_{\phi}$ of $\mathcal{L}_{\phi}$ such that $\mathcal{L}_{\phi}(h_{\phi}) =  e^{P(\phi, \sigma)} h_{\phi}$ and such that
\[
\int_{\Sigma_{\mathpzc{A}}} h_{\phi} \, d\mu_{\phi} \; = \; 1.
\]
\item The potential function $\phi$ has a unique equilibrium measure $\nu_{\phi}$.  Moreover, $\nu_{\phi}$ is given, for each $B \in \mathcal{B}$ (the Borel $\sigma$-algebra generated by $\mathcal{T}$, the set of cylinder sets of $\Sigma_{\mathpzc{A}}$), by
\[
\nu_{\phi}(B) \, \mathrel{:=} \, \int_{B} h_{\phi} \, d\mu_{\phi}.
\]
Hence, if $\phi, \psi \in C(\Sigma_{\mathpzc{A}}; \mathbb{C})$ are cohomologous, then their associated equilibrium measures are equal.
\item The unique equilibrium measure for the potential $\phi$, given in part $5$, is a Gibbs measure for the potential $\phi$.
\item (See also \textnormal{\cite{homologyatinfinity, Conformal_Erg}} and reference within.) The function $p \mathrel{:=} p_{\phi}: \mathbb{R} \to \mathbb{R}$ defined for each $t \in \mathbb{R}$, by $p(t) \mathrel{:=} P(t\phi)$, is a convex, real analytic function where
\[
p'(t) \; = \; \int_{\Sigma_{\mathpzc{A}}} \phi \, d\nu_{t \phi},
\quad
\text{and} 
\quad
p''(t) \;=\; \lim_{k \to \infty} \frac{1}{k} \Var_{\nu_{t\phi}}(S_{k}\phi).
\]
Here,
\[
\Var_{\nu_{t\phi}}(S_{k}\phi) \, \mathrel{:=} \, \int_{\Sigma_{\mathpzc{A}}} \left( S_{k}(\phi - E(\phi)) \right) \, d\nu_{t\phi}
\]
where $E(S_{k}(\phi)) \mathcal{:=} \int_{\Sigma_{\mathpzc{A}}} S_{k}(\phi) d\nu_{t\phi}$ denotes the expectation of $S_{k}(\phi)$ with respect to $\nu_{t\phi}$. 
\end{enumerate}
\end{thm}

\begin{rmk}
The limit $\lim_{k \to \infty} \Var_{\nu_{t\phi}}(S_{k}\phi)/k$ always exists and is called the \textit{asymptotic variance} of $(S_{k}\phi)$ with respect to $\nu_{t\phi}$, \cite{Conformal_Erg}.
\end{rmk}

\subsection{A Haar basis for subshifts of finite type}\label{TheHaarBasis}

In what follows, we develop an essential notion which will be required in Section \ref{MainSection}.  This notion is that of a generalisation of a Haar basis for the middle third Cantor set to the setting of a one-sided topologically exact subshift of finite type.    

Throughout this section let $(\Sigma_{\mathpzc{A}}, \sigma)$ denote a one-sided, topologically exact subshift of finite type and let $\mu_{\phi}$ denote a Gibbs measure for a H\"older continuous potential $\phi \in C(\Sigma_{\mathpzc{A}}; \mathbb{R})$.  Further, we make the following definitions and fix the following notation.

\begin{enumerate}
\item For each $k \in \mathbb{N}$ and each $x \mathrel{:=} (x_{1}, x_{2}, \dots, x_{k}) \in \Sigma_{\mathpzc{A}}^{k}$, fix a bijection
\[
\theta_{x}: \{  y \in  \Sigma \, : \, a_{x_{k}, y} = 1 \} \to \{ 1, 2, \dots, \alpha(x) \}.
\]
Observe that for each $x \in \Sigma_{\mathpzc{A}}^{*}$, the function $\theta_{x}$ places an ordering on the ``children'' of $x$.
\item For each $k \in \mathbb{N}$ and each $x \in \Sigma_{\mathpzc{A}}^{k}$, define the weighted inner product $\langle \cdot, \cdot \rangle_{x} : \mathbb{R}^{\alpha(x)} \times \mathbb{R}^{\alpha(x)} \to \mathbb{R}$ by
\[
\langle (r_{1}, r_{2}, \dots, r_{\alpha(x)}), (s_{1}, s_{2}, \dots, s_{\alpha(x)}) \rangle_{x} \; \mathrel{:=} \; \sum_{j = 1}^{\alpha(x)} \mu_{\phi}([x \theta_{x}^{-1}(j)]) r_{j}s_{j}
\]
and observe that the set
\begin{eqnarray*}
&& \{ f_{x, j} \; \mathrel{:=} \; (\mu_{\phi}([x \theta_{x}^{-1}(j)]))^{-\nicefrac{1}{2}} \, \delta_{j,\alpha(x)}: j \in \{ 1, 2, \dots, \alpha(x) \} \}
\end{eqnarray*}
forms an orthonormal basis for $(\mathbb{R}^{\alpha(x)}, \langle \cdot, \cdot \rangle_{x})$, where $\delta_{j,k}$ denotes the $j$-th standard unit vector in $\mathbb{R}^{k}$ and where $[x \theta_{x}^{-1}(j)]$ denotes the cylinder set $[(x_{1}, x_{2}, \dots, x_{k}, \theta_{x}^{-1}(j)) ]$, for $j \in \{ 1, 2, \dots, k\}$.
\item For each $x \in \Sigma^{*}_{\mathpzc{A}}$, let $\Omega_{x}$ denote the set defined by
\begin{align*}
\Omega_{x} \; \mathrel{:=}\; \{ U: \mathbb{R}^{\alpha(x)} \to \mathbb{R}^{\alpha(x)} \; : \;& \text{is linear and has positive determinant,}\\
& \langle U(\mathrm{v}), U(\mathrm{u}) \rangle_{x} = \langle \mathrm{v}, \mathrm{u} \rangle_{x}\; \text{for all} \; \mathrm{v}, \mathrm{u} \in \mathbb{R}^{\alpha(x)}\\
& \text{and} \; U(f_{x, \alpha(x)}) = (\mu_{\phi}([x]))^{-\nicefrac{1}{2}} (1, 1,\dots, 1)\}
\}
\end{align*}
and fix a family $\left( U_{x} \right)_{x \in \Sigma^{*}_{\mathpzc{A}}}$, where $U_{x} \in \Omega_{x}$.
\item For each $(x, j) \in  \bigcup_{y \in \Sigma_{\mathpzc{A}}^{*}} \{ y \}  \times \{ 1, 2, \dots, \alpha(y) - 1\}$, define $e_{x, j}: \Sigma_{\mathpzc{A}} \to \mathbb{R}$ by
\[
e_{x, j} \; \mathrel{:=} \; \sum_{k = 1}^{\alpha(x)} \, (\mu_{\phi}([x (\theta_{x}^{-1}(k))]))^{-\nicefrac{1}{2}} \, \langle f_{x, k}, U_{x}(f_{x, j}) \rangle_{x} \, \chi_{{[x (\theta_{x}^{-1}(k))]}}.
\]
Here, for each $y \in \Sigma_{\mathpzc{A}}^{*}$, we let $\chi_{{[y]}}$ denote the characteristic function on the cylinder set $[y]$.
\end{enumerate}
Following convention, throughout this article, for each $a \in C(\Sigma_{\mathpzc{A}}; \mathbb{C})$, we will also let $a$ denote, where appropriate, the equivalence class
\[
\left\{ f: \Sigma_{\mathpzc{A}} \to \mathbb{C} \, : \, f \; \text{is a measurable function and} \; \int_{\Sigma_{\mathpzc{A}}} \lvert f -a \rvert \, d\mu_{\phi} \, = \, 0 \right\}
\]
of $L^{2}(\Sigma_{\mathpzc{A}}, \mathcal{B}, \mu_{\phi})$, where $\mathcal{B}$ denotes the Borel $\sigma$-algebra generated by $\mathcal{T}$.  We will also view, without loss of generality, the elements of  $L^{2}:=L^{2}(\Sigma_{\mathpzc{A}}, \mathcal{B}, \mu_{\phi})$ as complex-valued functions.

\begin{rmk}
For each $k \in \mathbb{N} \setminus \{ 1 \}$ and each $x \in \Sigma_{A}^{k}$ there exists a canonical choice for $U_{x}$.  We construct this canonical choice in the following manner.  Assume that $\mathbb{R}^{k}$ is equipped with the standard Euclidean inner product, for $k \in \{ 2, 3, \dots, l\}$.  We will now show how to construct a (canonical) sequence of linear transformations $( V_{k})_{k = 1}^{l}$, where $V_{k}: \mathbb{R}^{k} \to \mathbb{R}^{k}$ and such that each $V_{k}$ satisfies the following.
\begin{enumerate}
\item $V_{k}$ is orientation preserving.
\item For every $x,y \in \mathbb{R}^{k}$, we have that $\langle V_{k}(x), V_{k}(y) \rangle = \langle x, y \rangle$.
\item $V_{k}(0, 0, \dots, 0, 1) = n^{-1/2} (1, 1, \dots, 1, 1)$.
\end{enumerate}
For $k = 2$, one has only one choice for $V_{2}$, namely
\[
V_{2} \mathrel{:=} \left( \begin{array}{cc} 2^{-1/2} & 2^{-1/2} \\ -2^{-1/2} & 2^{-1/2} \end{array} \right).
\]
For $k > 2$, assume that $V_{k - 1} \mathrel{:=} [v_{i, j}]_{i, j} : \mathbb{R}^{k - 1} \to \mathbb{R}^{k - 1}$ has been given, then define
\begin{align*}
\widetilde{V}_{k} &\mathrel{:=} \;
\begin{pmatrix}
v_{1, 1} & \dots & v_{1, k - 1} &  0\\
v_{2, 1} & \dots & v_{2, k - 1} &  0\\
\vdots & \ddots & \vdots & \vdots\\
v_{k-1, 1} & \dots & v_{k-1, k - 1} &  0\\
0 & \dots & 0 & 1
\end{pmatrix}
\intertext{and}
O_{k} &\mathrel{:=}\;
\begin{pmatrix}
1 & \dots & 0 & 0 & 0\\
\vdots & \ddots & \vdots & \vdots & \vdots\\
0 & \dots & 1 & 0 & 0\\
0 & \dots & 0 & k^{-1/2} & (1 - 1/k)^{1/2}\\
0 & \dots & 0 & -(1 - 1/k)^{1/2} & k^{-1/2}\\
\end{pmatrix} \in \mathcal{O}(k),
\end{align*}
where $\mathcal{O}(k)$ denotes the orthogonal group of degree $k$ over $\mathbb{R}$.  We then define $V_{k}$ to be the matrix $\widetilde{V}_{k}O_{k}\widetilde{V}_{k}^{t}$.  This allows us to define a canonical choice for $U_{x}$, namely the linear transformation $S^{-1} V_{\alpha(x)} S$, where
\[
S \; \mathrel{:=} \; \text{diag}\left( \mu_{\phi}([x \theta_{x}^{-1}(1)])^{1/2}, \dots, \mu_{\phi}([x \theta_{x}^{-1}(k)])^{1/2} \right).
\]
\end{rmk}

\begin{thm}\label{Lemma-Generlised-Haar-Basis}
The set
\begin{equation}\label{ONBhaar-defn}
\left\{ e_{x, j} \, : \, (x, j) \in {\textstyle{\bigcup_{y \in \Sigma_{\mathpzc{A}}^{*}}} \{ y \}  \times \{ 1, 2, \dots, \alpha(y) - 1\}} \right\} \; {\cup} \; \left\{ (\mu_{\phi}([x]))^{-\nicefrac{1}{2}} \chi_{{[x]}} \, : \, x \in \Sigma \right\}
\end{equation}
forms an orthonormal basis for $L^{2}(\Sigma_{\mathpzc{A}}, \mathcal{B}, \mu_{\phi})$.
\end{thm}

\begin{proof}
For each $x \in \Sigma$, we have that
\[
\lVert (\mu_{\phi}([x]))^{-\nicefrac{1}{2}} \chi_{{[x]}} \rVert^{2}_{{2}} \, = \, \int_{\Sigma_{\mathpzc{A}}} ( (\mu_{\phi}([x]))^{-\nicefrac{1}{2}}\chi_{{[x]}})^{2} d\mu_{\phi} \, = \, 1.
\]
For each $(x, j) \in  \bigcup_{y \in \Sigma_{\mathpzc{A}}^{*}} \{ y\} \times \{ 1, 2, \dots, \alpha(y) - 1\}$, we have that
\begin{eqnarray*}
\lVert e_{x, j} \rVert_{{2}}^{2} &=& \int_{\Sigma_{\mathpzc{A}}} \left( \sum_{k = 1}^{\alpha(x)} \, (\mu_{\phi}([x (\theta_{x}^{-1}(k))]))^{-\nicefrac{1}{2}} \, \langle f_{x, k}, U(f_{x, j}) \rangle_{x} \, \chi_{{[x (\theta_{x}^{-1}(k))]}} \right)^{2} d\mu_{\phi}\\
&=& \sum_{k = 1}^{\alpha(x)} \langle f_{x, k}, U_{x}(f_{x, j}) \rangle_{x}^{2}\\
&=& \lVert U_{x}(f_{x, j}) \rVert_{{\langle \cdot, \cdot \rangle_{{x}}}}^{2}\;= \; 1.
\end{eqnarray*}
In order to show that the set given in Equation (\ref{ONBhaar-defn}) is orthonormal we need to check the following cases.
\begin{enumerate}
\item For $x \in \Sigma$ and $(y, j) \in \bigcup_{u \in \Sigma_{\mathpzc{A}}^{*}} \{ u \}  \times \{ 1, 2, \dots, \alpha(u) - 1\}$ we need to show that $\langle (\mu_{\phi}([x]))^{-1} \chi_{{[x]}}, e_{y, j} \rangle_{{2}} = 0$.
\item  For $(x, j), (x, k) \in \bigcup_{y \in \Sigma_{\mathpzc{A}}^{*}} \{ y\} \times \{ 1, 2, \dots, \alpha(y) - 1\}$ with $j \not= k$ we need to show that $\langle e_{x, j}, e_{x, k} \rangle_{{2}} = 0$.
\item  For $(x, j), (y, k) \in  \bigcup_{u \in \Sigma_{\mathpzc{A}}^{*}} \{ u \} \times \{ 1, 2, \dots, \alpha(u) - 1\}$ with $x \not= y$ we need to show that $\langle e_{x, j}, e_{y, k} \rangle_{{2}} = 0$.
\end{enumerate}
Since the proofs of each of these cases are similar we provide a proof of Case $3$ and leave the other two case to the reader.  Suppose that $(x, j), (y, k) \in  \bigcup_{u \in \Sigma_{\mathpzc{A}}^{*}} \{ u \} \times \{ 1, 2, \dots, \alpha(u) - 1\}$ with $x \not= y$, then we have the following cases.
\begin{enumerate}
\item[(a)] If  $[x] \cap[y]=\varnothing$, then it is clear that $\langle e_{x, j}, e_{y, k} \rangle_{{2}} = 0$.
\item[(b)] If $[x] \subset [y]$, then there exists a constant $C \in \mathbb{R}$ such that
\begin{multline*}
\langle e_{x, j}, e_{y, k} \rangle_{{2}}\\
\begin{split}
= \;\;\;& C \, \int_{[y]} \sum_{m = 1}^{\alpha(x)} (\mu_{\phi}([x (\theta_{x}^{-1}(m))]) )^{-\nicefrac{1}{2}} \langle f_{x, m}, U_{x}(f_{x, j}) \rangle_{x} \chi_{{[x (\theta_{x}^{-1}(m))]}} \, d \mu_{\phi}\\
=\;\;\;& C \sum_{m = 1}^{\alpha(x)}  (\mu_{\phi}([x (\theta_{x}^{-1}(m))]) )^{\nicefrac{1}{2}} \langle f_{x, m}, U_{x}(f_{x,j}) \rangle_{x}\\
=\;\;\;& C (\mu([x]))^{\nicefrac{1}{2}} \langle U_{x}(f_{x, \alpha(x)}), U_{x}(f_{x, j}) \rangle_{x} \; = \; 0.
\end{split}
\end{multline*}
\item[(c)] If $[y] \subset [x]$, then using a symmetric argument to that given in Part $\mathrm{(b)}$, we have that $\langle e_{x, j}, e_{y, k} \rangle_{{2}} = 0$.
\end{enumerate}
By construction, every characteristic function of a cylinder set can be generated by a finite sum of elements of the set
\[
\left\{ e_{x, j} \, : \, (x, j) \in  {\textstyle{\bigcup}_{y \in \Sigma_{\mathpzc{A}}^{*}} \{ y \} \times \{ 1, 2, \dots, \alpha(y) - 1\}} \right\} \; {\cup} \; \left\{ (\mu_{\phi}([x]))^{-\nicefrac{1}{2}} \chi_{{[x]}} \, : \, x \in \Sigma \right\}.
\]
Therefore, the complex linear span of the above set contains the set of locally constant functions with finite range, which is $L^{2}$-norm dense in $L^{2}(\Sigma_{\mathpzc{A}}, \mathcal{B}, \mu_{\phi})$.  Hence, the result follows.
\end{proof}

\begin{defn}
We refer to the basis given in Theorem \ref{Lemma-Generlised-Haar-Basis} as a \textit{Haar basis for the measure space $(\Sigma_{\mathpzc{A}}, \mu_{\phi})$}.
\end{defn}

\subsection{Renewal Theorems in Symbolic Dynamics}

When investigating the noncommutative integration theory for our spectral metric space representations of Gibbs measures (Theorem \ref{Subshiftoffinitetypethmspectraltripleandotherintundmetricdim}) we will make use of counting results from symbolic dynamics.  These counting results are given by two different renewal theorems (Theorem \ref{RenewalthmLally1} and Theorem \ref{renewalthm}).

For stating the necessary results, we fix the following notation.  Let $\mu_{\phi}$ denote a Gibbs measure for a H\"older continuous potential function $\phi \in C(\Sigma_{\mathpzc{A}}; \mathbb{R})$.  For a finite word $x \in \Sigma_{\mathpzc{A}}^{*} \cup \{\emptyset\}$, we are interested in the asymptotic behaviour, as $r$ tends to zero, of the sums
 \begin{align}
\Upsilon_{x}(r) &\mathrel{:=} \; \sum_{\substack{y \in \Sigma_{\mathpzc{A}}^{*} \, \text{with} \\ \mu_{\phi}([y]) > r \; \text{and} \; [y] \subseteq [x]}} \hspace{-4mm} 1 \label{summandrewalLandF}
\intertext{and}
\Xi_{x}(r) &\mathrel{:=} \; \sum_{\substack{y \in \Sigma_{\mathpzc{A}}^{*} \, \text{with} \\ \mu_{\phi}([y]) > r \; \text{and} \; [y] \subseteq [x]}} \hspace{-4mm} \mu_{\phi}([y]).\label{summandrewalLandF1}
\end{align}
Here, when $x = \emptyset$, we set $[x] \mathrel{:=} \Sigma_{\mathpzc{A}}$.  

\begin{defn}
 For $f, g: \mathbb{R} \to \mathbb{R}$ and $x_{0}$ belonging to the extended real numbers, we say that $f$ is \textit{asymptotic} to $g$ as $t$ tends to $t_{0}$ (and in the case that $t_{0} = \pm\infty$, for all $t$ sufficiently large, respectively sufficiently small) if $\lim_{\, t \to t_{0}} \, f(t) / g(t) \, = \, 1$.  We write $f \sim g$ as $t$ tends to $t_{0}$.
\end{defn}

\begin{thm}\label{RenewalthmLally1}\textnormal{(Lalley \cite{Lalley})}
Let $(\Sigma_{\mathpzc{A}}, \sigma)$ denote a one-sided, topologically exact subshift of finite type and let $\phi_{1} \in C(\Sigma_{\mathpzc{A}})$ denote a H\"older continuous potential function with zero pressure.  Further, let $\phi_{2} \in C(\Sigma_{\mathpzc{A}})$ denote a non-negative but not identically zero H\"older continuous function.  Then there exists a strictly positive continuous function $C: \Sigma_{\mathpzc{A}} \to \Sigma_{\mathpzc{A}}$ such that, 
\[
 \sum_{k \in \mathbb{N}_{0}} \, \sum_{\upsilon \in \sigma^{-k}(\omega)} \phi_{2}(\upsilon) \, \chi_{{(-\infty, r]}} (S_{k}\phi_{1}(\upsilon)) \, \asymp \,  C(\omega) e^{r}.
\]
as $r \to \infty$, uniformly for $\omega \in \Sigma_{\mathpzc{A}}$.
\end{thm}

\begin{defn}
 For $f, g: \mathbb{R} \to [0, \infty)$ and $t_{0}$ belonging to the extended real numbers, we say that $f$ is \textit{comparable} to $g$ as $t$ tends to $t_{0}$ if there exist constants $c_{1}, c_{2} > 0$ such that for all $t$ sufficiently close to $t_{0}$ (and in the case that $t_{0} = \pm\infty$, for all $t$ sufficiently large, respectively sufficiently small), we have that $c_{{1}} f(x) \leq g(x) \leq c_{{2}} f(x)$. We write $f \asymp g$ as $t$ tends to $t_{0}$
\end{defn}

\begin{cor}\label{Cor1}
Under the conditions of Theorem \ref{RenewalthmLally1}, for each $x \in \Sigma_{\mathpzc{A}}^{*} \cup \{\emptyset\}$, we have that $\Upsilon_{x}(t) \asymp t^{-1}$ as $t$ tends to zero from above.
\end{cor}

\begin{proof}
Setting $\phi_{1} \mathrel{:=} \phi$ and $\phi_{2} \mathrel{:=} \chi_{{[x]}}$, the result follows from Theorem \ref{RenewalthmLally1}.
\end{proof}

 For the next theorem we will make use of a result by Krengel, R\"ottger and Wacker.

\begin{lem}\textnormal{(Krengel-R\"ottger-Wacker \cite{Krengel})}\label{KRW}
Let $Y_{1}, Y_{2}, \dots$ denote a stationary sequence of non-negative random variables with $\mathbb{E}(1/Y_{1}) < \infty$ and let $N_{t} \mathrel{:=} \sup \{ k \in \mathbb{N} : Y_{1} + Y_{2} + \dots + Y_{k} \leq t \}$.  Then $N_{t}/t$ converges in $L^{1}$-norm to $1/\lim_{k \rightarrow \infty} k^{-1}(Y_{1} + Y_{2} + \dots + Y_{k})$.
\end{lem}

\begin{thm}\label{renewalthm}
Let $(\Sigma_{\mathpzc{A}}, \sigma)$ denote a one-sided, topologically exact subshift of finite type and let $\nu_{\phi} \in \mathcal{M}_{\sigma}$ denote an equilibrium measure for a H\"older continuous potential function $\phi \in C(\Sigma_{\mathpzc{A}}; \mathbb{R})$ with zero pressure. Then for each $x \in \Sigma_{\mathpzc{A}}^{*} \cup \{ \emptyset \}$, as $r$ tends to infinity,
\[
\int_{[x]} \sum_{k \in \mathbb{N}_{0}} \chi_{[0, r]} (-S_{k}\phi(\omega)) \, d\nu_{\phi}(\omega) \, \sim \, \frac{r \nu_{\phi}([x])}{h_{\nu_{\phi}}(\sigma)}.
\]
\end{thm}

\begin{proof}
For a given $\phi \in C(\Sigma_{\mathpzc{A}}; \mathbb{C})$ It is well-known that there exists a strictly negative H\"older continuous potential function $\widetilde{\phi}$ with zero pressure cohomologous to $\phi$.  Therefore, by Theorem \ref{PFR-E-M-Uniq-thm}(5) and by the fact that $\lVert S_{k}\phi - S_{k}\widetilde{\phi} \rVert_{\infty}$ is bounded for all $k \in \mathbb{N}_{0}$, it is sufficient to show the result for $\phi$ a strictly negative H\"older continuous potential function with zero pressure.  The result then follows from Lemma \ref{KRW}  and Birkhoff's ergodic theorem, where, for $k \in \mathbb{N}$, the random variables $Y_{k}$ are defined to be equal to $- \phi \circ \sigma^{k-1}$ with respect to the equilibrium measure $\nu_{\phi}$.
\end{proof}

\begin{cor}\label{Cor2}
Let $(\Sigma_{\mathpzc{A}}, \sigma)$ denote a one-sided topologically exact subshift of finite type and let $\nu_{\phi}$ denote the unique equilibrium measure on $\Sigma_{\mathpzc{A}}$ for the H\"older continuous potential $\phi \in C(\Sigma_{\mathpzc{A}}; \mathbb{R})$.  Then, for each $x \in \Sigma_{\mathpzc{A}}^{*} \cup \{ \emptyset\}$, as $r$ tends to infinity, we have that
\[
\Xi_{x}(e^{-r}) \; \sim \; \frac{\nu_{\phi}([x]) r}{h_{\nu_{\phi}}(\sigma)}.
\]
\end{cor}

\begin{proof}
We may assume without loss of generality that the pressure of $\phi$ is equal to zero.  The result then follows by an application of Equation (\ref{variationPrinciple}), Theorem \ref{PFR-E-M-Uniq-thm} and Theorem \ref{renewalthm}.
\end{proof}

\section{Representing subshifts of finite type by spectral triples}\label{MainSection}

Althought the construction of a spectral metric space given in \cite{STCS} gives a noncommutative representation of a one-sided subshift of finite type, the measure theoretical properties of such a spectral triple essentially encodes the measure of maximal entropy.  In what follows, we consider a one-sided, topologically exact subshift of finite type $(\Sigma_{\mathpzc{A}}, \sigma)$ and an equilibrium measure $\nu_{\phi}$ for a H\"older continuous potential function $\phi \in C(\Sigma_{\mathpzc{A}}; \mathbb{C})$.   By refining Antonescu-Ivan's and Christensen's construction, we provide a spectral metric space which represents the metric space $(\Sigma_{\mathpzc{A}}, d)$ whose measure theoretical properties encode the measure $\nu_{\phi}$.  Indeed, by breaking down the projections used in Theorem $2.1$ of \cite{STCS} and by relating the singular values of the Dirac operator to the $\nu_{\phi}$-measure of the cylinder sets of $\Sigma_{\mathpzc{A}}$, we prove that a spectral triple $(A, H, D) \mathrel{:=} (C(\Sigma_{\mathpzc{A}}; \mathbb{C}), L^{2}(\Sigma_{\mathpzc{A}}, \mathcal{B}, \nu_{\phi}), D_{\nu_{\phi}})$ can be constructed so that the following hold.
\begin{enumerate}
\item Connes' pseudo-metric is a metric on $\mathcal{S}(A)$.  Moreover, the topology induced by Connes' pseudo metric on $\mathcal{S}(A)$ coincides with the weak-${*}$-topology defined on $\mathcal{S}(A)$.
\item $(A, H, D)$ is $(1, +)$-summable with metric dimension equal to one.
\item The noncommutative integral given by $(A, H, D)$ agrees with the integral with respect to $\nu_{\phi}$.
\item The noncommutative volume constant of $(A, H, D)$ is equal to $1 / h_{\nu_{\phi}}(\sigma)$.  
\end{enumerate}
We begin by fixing the following notation.  Let $\mu_{\phi}$ denote a Gibbs measure for some H\"older continuous potential function $\phi \in C(\Sigma_{\mathpzc{A}}; \mathbb{R})$ and let
\[
\left\{ (\mu_{\phi}([x]))^{-\nicefrac{1}{2}} \chi_{{[x]}} \, : \, x \in \Sigma \right\} \,  {\cup} \, \left\{ e_{x, j} \, : \, (x, j) \in \textstyle{\bigcup}_{y \in \Sigma_{\mathpzc{A}}^{*}} \{ y \}  \times \{ 1, 2, \dots, \alpha(y) -1 \} \right\}
\]
denote a Haar basis for $L^{2}(\Sigma_{\mathpzc{A}}, \mathcal{B}, \mu_{\phi})$, as described in Subsection \ref{TheHaarBasis}.  Let $\tau_{{\mu_{\phi}}}: C(\Sigma_{\mathpzc{A}}; \mathbb{C}) \to \mathbb{C}$ denote the tracial state defined, for each $a \in C(\Sigma_{\mathpzc{A}}; \mathbb{C})$, by
\[
\tau_{{\mu_{\phi}}}(a) \mathrel{:=} \int_{\Sigma_{\mathpzc{A}}} a \, d \mu_{\phi}.
\]
Then the Gelfand-Na\u{\i}mark-Segal completion of $C(\Sigma_{\mathpzc{A}}; \mathbb{C})$ with respect to $\tau_{{\mu_{\phi}}}$ is the Hilbert space $L^{2}(\Sigma_{\mathpzc{A}}, \mathcal{B}, \mu_{\phi})$.  Define the operator $D_{\mu_{\phi}}$  on $\text{Dom}(D_{\mu_{\phi}})$, a dense subset of $L^{2}(\Sigma_{\mathpzc{A}}, \mathcal{B}, \mu_{\phi})$, by
\begin{equation}
D_{\mu_{\phi}} \; \mathcal{:=} \; \sum_{x \in \Sigma} \frac{1}{\mu_{\phi}([x])} \langle \cdot, \chi_{{[x]}} \rangle \chi_{{[x]}} - \langle \cdot, \chi_{{\Sigma_{\mathpzc{A}}}} \rangle \chi_{{\Sigma_{\mathpzc{A}}}} \, + \sum_{y \in \Sigma^{*}_{\mathpzc{A}}} \frac{\alpha(y) - 1}{\mu_{\phi}([y])} \sum_{k = 1}^{\alpha(y) - 1} \langle \cdot, e_{y, k} \rangle  e_{y, k}.\label{ACDirac-adaptation}
\end{equation}
{Observe that the domain of $D_{\mu_{\phi}}$ contains all finite linear combinations of element of the given Haar basis, and thus $D_{\mu_{\phi}}$ is a densely defined operator.}

\begin{thm}\label{Subshiftoffinitetypethmspectraltripleandother}
The triple $(C(\Sigma_{\mathpzc{A}}; \mathbb{C}), L^{2}(\Sigma_{\mathpzc{A}}, \mathcal{B}, \pi, \mu_{\phi}), D_{\mu_{\phi}})$ is a spectral triple, where $\pi: C(\Sigma_{\mathpzc{A}}; \mathbb{C}) \to B(L^{2}(\Sigma_{\mathpzc{A}}, \mathcal{B}, \mu_{\phi}))$ is given by $\pi(a)h \, \mathrel{:=}\, a \cdot h$, for each $a \in C(\Sigma_{\mathpzc{A}}; \mathbb{C})$ and each $h \in  L^{2}(\Sigma_{\mathpzc{A}}, \mathcal{B}, \mu_{\phi})$.
\end{thm}

\begin{proof}
It is clear to see that the set $C(\Sigma_{\mathpzc{A}}; \mathbb{C})$ equipped with the supremum norm is a $C^{*}$-algebra, that $L^{2}(\Sigma_{\mathpzc{A}}, \mathcal{B}, \mu_{\phi})$ is a complex Hilbert space and that $(\pi, L^{2}(\Sigma_{\mathpzc{A}}, \mathcal{B}, \mu_{\phi}))$ is a faithful $*$-representation of $C(\Sigma_{\mathpzc{A}}; \mathbb{C})$.  Further, we have that the domain of $D_{\mu_{\phi}}$ contains all finite linear combinations of element of the given Haar basis, and thus $D_{\mu_{\phi}}$ is a densely defined operator.  Observe that by the properties of a Gibbs measure (Equation (\ref{GibbPropertyMeasureDefn})), the operator $(1 + D_{\mu_{\phi}}^{2})^{-\nicefrac{1}{2}}$ is a bounded operator on $L^{2}(\Sigma_{\mathpzc{A}}, \mathcal{B}, \mu_{\phi})$ which can be approximated by operators in $B( L^{2}(\Sigma_{\mathpzc{A}}, \mathcal{B}, \mu_{\phi})$ with finite dimensional range.  Therefore, $D_{\mu_{\phi}}$ has a compact resolvent.  Moreover, the sets $\text{Ran}(D_{\mu_{\phi}} \pm i 1)$ are $L^{2}$-norm-dense in $L^{2}(\Sigma_{\mathpzc{A}}, \mathcal{B}, \mu_{\phi})$.  Indeed, this follows since $D$ is a linear operator and since we have the following.
\begin{enumerate}
\item For each $x \in \Sigma$, we have that
\[
(D_{\mu_{\phi}} \pm i 1)\left( \frac{1}{1 \pm i} (\chi_{{[x]}}) \mp i \mu_{\phi}([x]) \chi_{{\Sigma_{\mathpzc{A}}}} \right) \, = \, \chi_{{[x]}}.
\]
\item For each $(x, j) \in \bigcup_{y \in \Sigma_{\mathpzc{A}}^{*}} \{ y \} \times \{ 1, 2, \dots, \alpha(y) - 1\}$, we have that
\[
(D_{\mu_{\phi}} \pm i1)\left( \frac{\mu_{\phi}([x])}{\alpha(x) - 1 \pm i} \, e_{x, j} \right) \, = \, e_{x, j}.
\]
\end{enumerate}
Moreover, we have that $D_{\mu_{\phi}}$ is symmetric on its domain.  Indeed, for each $h_{1}, h_{2} \in \text{Dom}(D_{\mu_{\phi}}) \subset L^{2}(\Sigma_{\mathpzc{A}}, \mathcal{B}, \mu_{\phi})$, a short calculation using the fact that $L^{2}(\Sigma_{\mathpzc{A}}, \mathcal{B}, \mu_{\phi})$ is separable, and hence, that every element can be expressed as a linear (possibly infinite) combination of basis elements will show that $\langle D_{\mu_{\phi}}(h_{1}), h_{2} \rangle = \langle h_{1}, D_{\mu_{\phi}}(h_{2}) \rangle$.  Hence, by  \cite[Corollary to Theorem $\mathrm{VIII}.3$]{simonreed}, it follows that $D_{\mu_{\phi}}$ is an essentially self-adjoint operator.  Therefore, in order to show that $(C(\Sigma_{\mathpzc{A}}; \mathbb{C}), L^{2}(\Sigma_{\mathpzc{A}}, \mathcal{B}, \mu_{\phi}),$ $D_{\mu_{\phi}})$ is a spectral triple, all that remains to show is that the set
\begin{align*}
\mathcal{A} \, \mathrel{:=} \, \{ a \in C(\Sigma_{\mathpzc{A}}; \mathbb{C}) \; : \;&  [D_{\mu_{\phi}}, \pi(a)] \; \text{is densely defined and extends to}\\
& \text{a bounded operator defined on} \;  L^{2}(\Sigma_{\mathpzc{A}}, \mathcal{B}, \mu_{\phi}) \}
\end{align*}
is norm-dense in $C(\Sigma_{\mathpzc{A}}; \mathbb{C})$ with respect to the supremum norm.  To this end, for each $k \in \mathbb{N}$, consider the finite dimensional algebra $\mathcal{C}_{k}$ consisting of locally constant complex-valued functions spanned by the set $\{ \chi_{[x]} : x \in \Sigma_{\mathpzc{A}} \}$ and let $H_{k}$ denote the subspace of $\mathcal{L}^{2}(\Sigma_{\mathpzc{A}}, \mathcal{B}, \mu_{\phi})$ spanned by the functions in $\mathcal{C}_{k}$.  Observe that $\bigcup_{k \in \mathbb{N}} \mathcal{C}_{k}$ is dense under the supremum norm in $C(\Sigma_{\mathpzc{A}}; \mathbb{C})$ and thus the result will follow, if for an arbitrary $k \in \mathbb{N}$ and for any $f \in \mathcal{C}_{k}$, the commutator $[D_{\mu_{\phi}}, \pi(f)]$ is defined and bounded.  Further, observe that by the construction of the Haar basis, for $k \in \mathbb{N}$, the subspace $H_{k + 1}$ is spanned by the set
\[
\left\{ e_{x, j} : (x, j) \in \bigcup_{m =1}^{k} \bigcup_{y \in \Sigma_{\mathpzc{A}}^{m}} \{ y \} \times \{ 1, 2, \dots, \alpha(y) \} \right\} \, \textstyle{\bigcup} \, \left\{ \chi_{[y]} : y \in \Sigma_{\mathpzc{A}} \right\}
\]
Next, let $\beta \mathrel{:=} \sup \{ \phi(\omega) : \omega \in \Sigma_{\mathpzc{A}} \} - P(\phi, \sigma) < \infty$, let $c >1$ be the constant given in Theorem \ref{P-FO-} and define $m_{k} \mathrel{:=} c e^{-k \beta}$, for each $k \in \mathbb{N}$.  By construction $D_{\mu_{\phi}}\vert_{H_{k+1}}$ is bounded with $\lVert D_{\mu_{\phi}}\vert_{H_{k+1}} \rVert \leq m_{k}$.  Further, for all $f \in \mathcal{C}_{k}$ and $h \in H_{k + 1}$, it is clear to see that $h \in \text{Dom}([D_{\mu_{\phi}}, \pi(f)])$ and that $\lVert [D_{\mu_{\phi}}, \pi(f)] h \rVert \leq 2 m_{k} \lVert f \rVert_{\infty} \lVert h \rVert_{2}$.  Next we observe, for $n > k$, $y \in \Sigma_{\mathpzc{A}}^{n}$, $i \in \{ 1, 2, \dots, \alpha(y) - 1 \}$ and $x \in \Sigma^{k}$, that $[D_{\mu_{\phi}}, \pi(\chi_{[x]})] e_{y, i} = 0$ and hence, $[D_{\mu_{\phi}}, \pi(f)] e_{y, i} = 0$, for all $f \in \mathcal{C}_{k}$.  Therefore, for any $f \in \mathcal{C}_{k}$, the commutator $[D_{\mu_{\phi}}, \pi(f)]$ is defined and bounded by $2 m_{k} \lVert f \rVert_{\infty}$ on $\bigcup_{n \in \mathbb{N}} H_{n}$, which is norm dense in $L^{2}(\Sigma_{\mathpzc{A}}, \mathcal{B}, \mu_{\phi})$.  
\end{proof}

\subsection{Metric aspects}

The spectral triple given in Theorem \ref{Subshiftoffinitetypethmspectraltripleandother} should be able to recover some of the geometric structure of the original space $\Sigma_{\mathpzc{A}}$.  The following theorem shows
that this is the case.

\begin{thm}\label{thm-connes-metric-adapt-dirac-AC}
Letting $(A, H, D) \mathrel{:=} (C(\Sigma_{\mathpzc{A}}; \mathbb{C}), L^{2}(\Sigma_{\mathpzc{A}}, \mathcal{B}, \mu_{\phi}), D_{\mu_{\phi}})$, we have that Connes' pseudo-metric $d$ is a metric on $\mathcal{S}(A)$.  Moreover, the topology induced by $d$ on $\mathcal{S}(A)$ coincide with the weak-${*}$-topology defined on $\mathcal{S}(A)$.  Hence, the spectral triple $(A, H, D)$ is a spectral metric space.
\end{thm}

\begin{proof}
The proof of this result is motivated by the proof of \cite[Theorem $2.1$]{STCS}.  For ease of notation, set $\mathrm{card}(\Sigma) \mathrel{=:} l \in \mathbb{N}$ and set
\begin{eqnarray*}
\mathcal{A}_{{D_{\mu_{\phi}}}} \, \mathrel{:=}\, \{ a \in A &\hspace{-4mm}:& \hspace{-3mm} [D_{\mu_{\phi}}, \pi(a)] \; \text{is densely defined and extends to a bounded}\\
&& \hspace{-3mm} \text{operator defined on $H$ with norm less than or equal to one} \}.
\end{eqnarray*}
Our aim is to show that there exists a constant $C$ such that $\lVert a - \int_{\Sigma_{\mathpzc{A}}} a \, d\mu_{\phi} \rVert_{\infty} \leq C$, for all $a \in \mathcal{A}_{D_{\mu_{\phi}}}$.  Since if this was the case, then the quotient norm of $a + \{ z \chi_{{\Sigma_{\mathpzc{A}}}} \, : \, z \in \mathbb{C} \}$ in the quotient space $\mathcal{A}_{{D_{\mu_{\phi}}}} / \{ z \chi_{{\Sigma_{\mathpzc{A}}}} \, : \, a \in A \}$ is bounded above by $C$.  Therefore, by the first part of Theorem \ref{RieffelandPavlovicThmNCG}, Connes' pseudo-metric $d$ is a metric.
To this end, let  $\pi_{0}, \pi_{1}: A \to A$ denote the projections given, for each $a \in A$, by 
\[
\pi_{0}(a) \mathrel{:=} \int_{\Sigma_{\mathpzc{A}}} a \, d\mu_{\phi} \, \chi_{{\Sigma_{\mathpzc{A}}}}, \quad \text{and} \quad
\pi_{1}(a) \mathrel{:=} \sum_{x \in \Sigma} (\mu_{\phi}([x]))^{-1} \int_{[x]} a \, d\mu_{\phi} \, \chi_{{[x]}}.
\]
Further, for each $k \in \mathbb{N}$, let $\pi_{k + 1}: A \to A$ denote the projection given, for each $a \in A$ by
\begin{multline*}
\hspace{-3.75mm}\pi_{k + 1}(a)\\
\begin{split}
\mathrel{:=}\;\;& \sum_{x \in \Sigma} (\mu_{\phi}([x]))^{-1} \int_{[x]} a \, d\mu_{\phi} \, \chi_{{[x]}} \, + \, \sum_{m = 1}^{k} \sum_{x \in \Sigma_{\mathpzc{A}}^{m}} \sum_{j = 1}^{\alpha(x) - 1} \int_{\Sigma_{\mathpzc{A}}} a \cdot e_{x, j} \, d\mu_{\phi} \, e_{x, j}.
\end{split}
\end{multline*}
Suppose that there exists a constants $C' > 0$ and $\beta < 0$ such that
\begin{align}\label{Cauchy_Norm_Bound}
\lVert \pi_{k}(a) - \pi_{k + m}(a) \rVert_{\infty} \leq C' \sum_{n = k}^{k +m -1} e^{n \beta}
\end{align}
for all $k, m \in \mathbb{N}$ and $a \in \mathcal{A}_{D_{\mu_{\phi}}}$.  Then $( \pi_{k}(a) )_{k \in \mathbb{N}}$ is a Cauchy sequence in $A$, with respect to the supremum norm.  Therefore, the sequence $(\pi_{k}(a) )_{k \in \mathbb{N}}$ is convergent.  Let $b \in A$ denote its limit and observe that 
\begin{align}\label{eq_a=b}
\pi(b)\chi_{{\Sigma_{\mathpzc{A}}}} \, = \, \lim_{k \to \infty} \pi(\pi_{k}(a)) \chi_{{\Sigma_{\mathpzc{A}}}} \, = \, \lim_{k \to \infty} \, \sum_{m = 1}^{k} \, P_{m} \pi(a) \chi_{{\Sigma_{\mathpzc{A}}}} \, = \, \pi(a) \chi_{{\Sigma_{\mathpzc{A}}}}.
\end{align}
Here $P_{1}: H \to H$ denotes the projections given by
\[
P_{1} \;\mathrel{:=} \; \sum_{x \in \Sigma} (\mu_{\phi}([x]))^{-1} \langle \cdot , \chi_{{[x]}} \rangle \chi_{{[x]}},
\]
and, for each $k \in \mathbb{N}$, we let $P_{k + 1}: H \to  H$ denote the projection given by
\[
P_{k + 1} \mathrel{:=} \sum_{x \in \Sigma_{\mathpzc{A}}^{k}} \sum_{j = 1}^{\alpha(x) - 1} \langle \cdot, e_{x, j} \rangle e_{x, j}.
\]
Since $\chi_{{\Sigma_{\mathpzc{A}}}}$ is a separating and cyclic vector for the subalgebra $\pi(A)$ in $B(H)$, it follows from Equation (\ref{eq_a=b}) that $b = a$.   {Now, for each $a \in \mathcal{A}_{D_{\mu_{\phi}}}$, since $D_{\mu_{\phi}} \chi_{\Sigma_{\mathpzc{A}}} = 0$, we have that
\begin{align*}
1 \geq \lVert \overline{[D_{\mu_{\phi}}, \pi(a)]} \rVert \geq& \; {\lVert P_{1} \overline{[D_{\mu_{\phi}}, \pi(a)]} \rVert}\nonumber\\
{=}& \; \lVert P_{1} [D_{\mu_{\phi}}, \pi(a)] \rVert\nonumber\\
\geq& \; \lVert P_{1} [D_{\mu_{\phi}}, \pi(a)] \chi_{\Sigma_{\mathpzc{A}}} \rVert_{2}\nonumber\\
=& \; \lVert P_{1} D_{\mu_{\phi}} \pi(a) \chi_{\Sigma_{\mathpzc{A}}} - P_{1}\pi(a)D_{\mu_{\phi}} \chi_{\Sigma_{\mathpzc{A}}}\rVert_{2}\nonumber\\
=& \; \lVert P_{1} D_{\mu_{\phi}} \pi(a) \chi_{\Sigma_{\mathpzc{A}}} \rVert_{2},
\end{align*}
where  $\overline{[D_{\mu_{\phi}}, \pi(a)]} \in B(L^{2}(\Sigma_{\mathpzc{A}}, \mathcal{B}, \mu_{\phi}))$ denotes the continuous extension of the densely defined operator $[D_{\mu_{\phi}}, \pi(a)]$.  Further, by the Gibbs property of $\mu_{\phi}$, as stated in Equation (\ref{GibbPropertyMeasureDefn}), and since $\mu_{\phi}$ is a probability measure, we have $0 < \mu_{\phi}([x]) < 1$, for each $x \in \Sigma$. Therefore,
\begin{align*}
1 &\geq  \lVert P_{1} D_{\mu_{\phi}} \pi(a) \chi_{\Sigma_{\mathpzc{A}}} \rVert_{2}\\
&=  \left( \langle P_{1}D_{\mu_{\phi}} \pi(a) \chi_{\Sigma_{\mathpzc{A}}}, P_{1}D_{\mu_{\phi}} \pi(a) \chi_{\Sigma_{\mathpzc{A}}} \rangle_{2} \right)^{1/2}\\
&= \left( \left( \int_{\Sigma_{\mathpzc{A}}} a \, d\mu_{\phi} \right)^{2} + \sum_{x \in \Sigma} \frac{1}{\mu_{\phi}([x])^{2}} \left( \int_{[x]} a \, d\mu_{\phi} \right)^{2} - \sum_{x \in \Sigma} \int_{\Sigma_{\mathpzc{A}}} a \, d\mu \int_{[x]} a \, d\mu \right)^{1/2}\\
&=  \left(\sum_{x \in \Sigma} \frac{1}{\mu_{\phi}([x])^{2}} \left( \int_{[x]} a \, d\mu_{\phi} \right)^{2} \right)^{1/2}\\
&\geq { \frac{1}{l^{1/2}}}\sum_{x \in \Sigma} \frac{1}{\mu_{\phi}([x])} \left\lvert \int_{[x]} a \, d\mu_{\phi} \right\rvert\\
&\geq {  \frac{1}{2\cdot l^{1/2}}} \left(\left\lvert \int_{\Sigma_{\mathpzc{A}}} a \, d\mu_{\phi} \right\rvert + \sum_{x \in \Sigma} \frac{1}{\mu_{\phi}([x])} \left\lvert \int_{[x]} a \, d\mu_{\phi} \right\rvert \right)\\
&\geq {  \frac{1}{2\cdot l^{1/2}}} \sup \left\{ \left\lvert  \int_{\Sigma_{\mathpzc{A}}} a \, d\mu_{\phi} \right\rvert + \frac{1}{\mu_{\phi}([x])} \left\lvert \int_{[x]} a \, d\mu_{\phi}  \right\rvert : x \in \Sigma \right\}\\
&\geq {  \frac{1}{2\cdot l^{1/2}}} \sup \left\{ \left\lvert  \int_{\Sigma_{\mathpzc{A}}} a \, d\mu_{\phi} - \frac{1}{\mu_{\phi}([x])}  \int_{[x]} a \, d\mu_{\phi}  \right\rvert : x \in \Sigma \right\}\\
&= {  \frac{1}{2\cdot l^{1/2}}} \sup \left\{ \left\lvert  \int_{\Sigma_{\mathpzc{A}}} a \, d\mu_{\phi} \, \chi_{{\Sigma_{\mathpzc{A}}}}(\omega) - \sum_{x \in \Sigma} \frac{1}{\mu_{\phi}([x])} \int_{[x]} a \, d\mu_{\phi} \, \chi_{{[x]}}(\omega) \right\rvert : \omega \in \Sigma_{\mathpzc{A}} \right\}\\
&= { \frac{1}{2\cdot l^{1/2}}} \left\lVert \int_{\Sigma_{\mathpzc{A}}} a \, d\mu_{\phi} \, \chi_{{\Sigma_{\mathpzc{A}}}} - \sum_{x \in \Sigma} (\mu_{\phi}([x]))^{-1} \int_{[x]} a \, d\mu_{\phi} \, \chi_{{[x]}} \right\rVert_{\infty}\\
&= { \frac{1}{2\cdot l^{1/2}}} \lVert \pi_{0}(a) - \pi_{1}(a) \rVert_{\infty}.
\end{align*}}
(Recall that $l \mathrel{:=} \text{card}(\Sigma)$.)  Thus, for each $k \in \mathbb{N}$ and $a \in \mathcal{A}_{D_{\mu_{\phi}}}$, we have that
\begin{align*}
\lVert \pi_{0}(a) - \pi_{k}(a) \rVert_{\infty} &\leq \lVert \pi_{0}(a) - \pi_{1}(a) \rVert_{\infty} + \lVert \pi_{1}(a) - \pi_{k}(a) \rVert_{\infty}\\
&\leq 2\cdot l^{1/2} + C'(1 - e^{\beta/2})^{-1}
\end{align*}
Letting $k$ tend to infinity,  we then have
\[
\lVert \pi_{0}(a) - a \rVert_{\infty} \, \leq \, 2\cdot l^{1/2} + C' (1 - e^{\beta/2})^{-1},
\]
for each $a \in \mathcal{A}_{{D_{\mu_{\phi}}}}$.  Hence, it remains is to show that there exists constants $C' > 0$ and $\beta < 0$ which satisfy Equation (\ref{Cauchy_Norm_Bound}).  To this end, recall that $h_{\phi}$ denotes the unique eigenfunction of $\mathcal{L}_{\phi}$ given by Theorem \ref{PFR-E-M-Uniq-thm}(4). Then $\widetilde{\phi}:=\phi -P(\phi, \sigma)+\log{h_{\phi}}-\log h_{\phi}\circ\sigma \in C(\Sigma_{\mathpzc{A}}; \mathbb{C})$ defines the strictly negative normalized  H\"older continuous potential function with zero pressure cohomologous to $\phi$.  By Theorem \ref{PFR-E-M-Uniq-thm}(5) the measures $\mu_{\phi}$ and $\mu_{\widetilde{\phi}}$ are equal.    Since $\mu_{\phi}$ satisfies the Gibbs property as stated in Equation (\ref{GibbPropertyMeasureDefn}), for some $c > 1$, we have
\begin{align*}
\mu_{\phi}([\omega_{1}, \omega_{2}, \dots, \omega_{k}]) &\leq c^{2} e^{S_{k}\widetilde{\phi}(\omega) - S_{k+1}\widetilde{\phi}(\omega)} \mu_{\phi}([\omega_{1}, \omega_{2}, \dots, \omega_{k + 1}])\\
&\leq c^{2} e^{-\gamma} \mu_{\phi}([\omega_{1}, \omega_{2}, \dots, \omega_{k + 1}])
\end{align*}
for each $\omega \mathrel{:=} (\omega_{1}, \omega_{2}, \dots ) \in \Sigma_{\mathpzc{A}}$ and $k \in \mathbb{N}$, where $\gamma \mathrel{:=}  \inf \{ \widetilde{\phi}(u) : u \in \Sigma_{\mathpzc{A}} \}$.  Set $\beta = \sup \{ \widetilde{\phi}(u)/2 : u \in \Sigma_{\mathpzc{A}} \}$ and observe that since $\widetilde{\phi}$ is continuous and strictly negative and since $\Sigma_{\mathpzc{A}}$ is compact, $\beta$ is finite and strictly less than zero.  Further, by the Gibbs property of $\mu_{\phi}$, we have $c_{k} \, \mathrel{:=} \, \sup\left\{ (\mu_{\phi}([x])) \,: \, x \in \Sigma_{\mathpzc{A}}^{k} \right\}\leq c e^{2 k\beta}$, for each $k \in \mathbb{N}$.  Hence, for each $k, m \in \mathbb{N}$ and $a \in \mathcal{A}_{D_{\mu_{\phi}}}$, we have 
\begin{align*}
&\lVert \pi_{k}(a) - \pi_{k + m}(a) \rVert_{\infty} \\
&= \; \sup_{\omega \in \Sigma_{\mathpzc{A}}} \left\lvert \sum_{n = k}^{k + m -1} \sum_{x \in \Sigma_{\mathpzc{A}}^{n}} \sum_{j = 1}^{\alpha(x) - 1} \int_{\Sigma_{\mathpzc{A}}} a \, e_{x, j} \, d\mu_{\phi} \, e_{x, j}(\omega) \right\rvert \\
&\leq \; \sup_{\omega \mathrel{:=} (\omega_{1}, \omega_{2}, \dots) \in \Sigma_{\mathpzc{A}}} \sum_{n = k}^{k + m -1} \sum_{x \in \Sigma_{\mathpzc{A}}^{n}} \sum_{j = 1}^{\alpha(x) - 1} \lvert \langle a, e_{(\omega_{1}, \omega_{2}, \dots, \omega_{n}), j} \rangle \rvert \lvert e_{(\omega_{1}, \omega_{2}, \dots, \omega_{n}), j}(\omega) \rvert \\
&\leq \; \sup_{\omega \mathrel{:=} (\omega_{1}, \omega_{2}, \dots) \in \Sigma_{\mathpzc{A}}} \sum_{n = k}^{k + m -1} { \sum_{j = 1}^{\alpha(\omega_{1}, \omega_{2}, \dots, \omega_{n}) - 1}} \lvert \langle a, e_{(\omega_{1}, \omega_{2}, \dots, \omega_{n}), j} \rangle \rvert \lvert e_{(\omega_{1}, \omega_{2}, \dots, \omega_{n}), j}(\omega) \rvert\\
&\leq \; c e^{\gamma/2} \sup_{  (\omega_{1}, \omega_{2}, \dots) \in \Sigma_{\mathpzc{A}}} \sum_{n = k}^{k + m -1} { \sum_{j = 1}^{\alpha(\omega_{1}, \omega_{2}, \dots, \omega_{n}) - 1}}\frac{\lvert \langle a, e_{(\omega_{1}, \omega_{2}, \dots, \omega_{n}), j} \rangle\rvert}{\mu_{\phi}([(\omega_{1}, \omega_{2}, \dots, \omega_{n}])^{1/2}} \\
&\leq \; c e^{\gamma/2} \sup_{ (\omega_{1}, \omega_{2}, \dots) \in \Sigma_{\mathpzc{A}}} \sum_{n = k}^{k + m -1} c_{n}^{1/2} { \sum_{j = 1}^{\alpha(\omega_{1}, \omega_{2}, \dots, \omega_{n}) - 1}} \frac{\lvert \langle a, e_{(\omega_{1}, \omega_{2}, \dots, \omega_{n}), j} \rangle\rvert}{\mu_{\phi}([(\omega_{1}, \omega_{2}, \dots, \omega_{n}])} \\
&\leq \; c e^{\gamma/2} \sup_{ (\omega_{1}, \omega_{2}, \dots) \in \Sigma_{\mathpzc{A}}} \sum_{n = k}^{k + m -1} e^{n \beta}{ \sum_{j = 1}^{\alpha(\omega_{1}, \omega_{2}, \dots, \omega_{n}) - 1}} \frac{\lvert \langle a, e_{(\omega_{1}, \omega_{2}, \dots, \omega_{n}), j} \rangle\rvert}{\mu_{\phi}([(\omega_{1}, \omega_{2}, \dots, \omega_{n}])}.
\end{align*} 
Recall that  $\overline{[D_{\mu_{\phi}}, \pi(a)]} \in B(L^{2}(\Sigma_{\mathpzc{A}}, \mathcal{B}, \mu_{\phi}))$ denotes the continuous extension of the densely defined operator $[D_{\mu_{\phi}}, \pi(a)]$.  Now since $D_{\mu_{\phi}} \chi_{\Sigma_{\mathpzc{A}}} = 0$, 
we have
\begin{align*}
1 \geq& \; \lVert \overline{[D_{\mu_{\phi}}, \pi(a)]} \rVert^{2}\\
\geq& \; {\lVert P_{k + 1}\overline{[D_{\mu_{\phi}}, \pi(a)]} \rVert^{2}}\\
{=}& \; \lVert P_{k + 1}[D_{\mu_{\phi}}, \pi(a)] \rVert^{2}\\
\geq& \; \lVert P_{k + 1}[D_{\mu_{\phi}}, \pi(a)] \chi_{\Sigma_{\mathpzc{A}}}\rVert_{2}^{2} \\
\geq& \; \lVert P_{k + 1} D_{\mu_{\phi}} \pi(a)\chi_{\Sigma_{\mathpzc{A}}} - P_{k+1}\pi(a)D_{\mu_{\phi}} \chi_{\Sigma_{\mathpzc{A}}} \rVert_{2}^{2}\\
\geq& \; \lVert P_{k + 1} D_{\mu_{\phi}} \pi(a)\chi_{\Sigma_{\mathpzc{A}}} \rVert_{2}^{2}\\
=& \; \sum_{x \in \Sigma_{\mathpzc{A}}^{k}} \sum_{j = 1}^{\alpha(x) - 1} \frac{1}{\mu_{\phi}([x])^{2}} \lvert \langle a, e_{x, j} \rangle\rvert^{2}\\
\geq& \; { \sum_{j = 1}^{\alpha(\omega_{1}, \omega_{2}, \dots, \omega_{n}) - 1}} \frac{1}{\mu_{\phi}([(\omega_{1}, \omega_{2}, \dots, \omega_{k})])^{2}} \lvert \langle a, e_{(\omega_{1}, \omega_{2}, \dots, \omega_{k}), j} \rangle\rvert^{2}\\
\geq& \; \frac{1}{l} \left( \sum_{j = 1}^{\alpha(\omega_{1}, \omega_{2}, \dots, \omega_{n}) - 1} \frac{1}{\mu_{\phi}([(\omega_{1}, \omega_{2}, \dots, \omega_{k})])} \lvert \langle a, e_{(\omega_{1}, \omega_{2}, \dots, \omega_{k}), j} \rangle\rvert \right)^{2},
\end{align*}
for each $a \in \mathcal{A}_{D_{\mu_{\phi}}}$ $k \in \mathbb{N}$ and $\omega \mathrel{:=} (\omega_{1}, \omega_{2}, \dots) \in \Sigma_{\mathpzc{A}}$.  Therefore, by setting $C' = l^{1/2}ce^{\gamma/2}$, we have that
\begin{align*}
\lVert \pi_{k}(a) - \pi_{k + m}(a) \rVert_{\infty} \leq C' \sum_{n = k}^{k + m -1} e^{n \beta}.
\end{align*}

To see that the topology induced by $d$ on $\mathcal{S}_{A}$ coincides with the weak-$*$-topology  we employ the  second part of Theorem \ref{RieffelandPavlovicThmNCG} and show that the  image of $\mathcal{A}_{{D_{\mu_{\phi}}}}$ under the quotient map $A \to A /  \{ z \chi_{{\Sigma_{\mathpzc{A}}}} \, : \, z \in \mathbb{C} \}$ is totally bounded. In fact,  this is an immediate consequence of  the following two  properties, which are easily deduced from the
above observations.
\begin{enumerate}
\item For each $\epsilon > 0$, there exists $k \in \mathbb{N}$ such that for any $m \in \mathbb{N}_{0}$ and for each $a \in \mathcal{A}_{{D_{\mu_{\phi}}}}$, we have that $\lVert a - \pi_{k + m}(a) \rVert_{\infty} \, < \, \epsilon$.
\item For each $k \in \mathbb{N}$, the space $\pi_{k}(A)$ is finite dimensional, and so, the closed ball of radius $2\cdot l^{1/2} + C' (1 - e^{\beta/2})^{-1}$ in $\pi_{k}(A)$ is norm compact.
\end{enumerate}
\end{proof}

\subsection{The metric dimension and noncommutative integral}

In the following theorem we consider the metric dimension and measure theoretical aspects of the spectral triple $(C(\Sigma_{\mathpzc{A}}; \mathbb{C}), L^{2}(\Sigma_{\mathpzc{A}}, \mathcal{B}, \mu_{\phi}), D_{\mu_{\phi}})$, as given in Theorem \ref{Subshiftoffinitetypethmspectraltripleandother}.

\begin{thm}\label{Subshiftoffinitetypethmspectraltripleandotherintundmetricdim}
Let $\phi \in C(\Sigma_{\mathpzc{A}}; \mathbb{R})$ denote a H\"older continuous potential function and let $\nu_{\phi}$ denote the unique equilibrium measure for the potential $\phi$.  Then the spectral triple
\[
(A, H, D) \, \mathrel{:=} \,(C(\Sigma_{\mathpzc{A}}; \mathbb{C}), L^{2}(\Sigma_{\mathpzc{A}}, \mathcal{B}, \nu_{\phi}), D_{\nu_{\phi}})
\]
is $(1, +)$-summable.  Moreover, for each $a \in C(\Sigma_{\mathpzc{A}}; \mathbb{C})$, we have that
\begin{equation}\label{intsubshiftfinitetypenon}
\Nint \pi(a) \lvert D_{\nu_{\phi}} \rvert^{-1} \, = \,\frac{1}{h_{\nu_{\phi}}(\sigma)} \, \int_{\Sigma_{\mathpzc{A}}} a \, d \nu_{\phi}
\end{equation}
and hence $(A, H, D)$ has metric dimension equal to one.
\end{thm}

\begin{proof}
For each $x \in \Sigma_{\mathpzc{A}}^{*} \cup \{\emptyset\}$, let $\Upsilon_{x}, \Xi_{x}: (0, \infty) \to [0, \infty)$ denote the functions that are respectively defined in Equations (\ref{summandrewalLandF}) and (\ref{summandrewalLandF1}).  Let $r \in (0, 1)$ be fixed and let $\mathrm{card}(\Sigma) \mathrel{=:} M \in \mathbb{N}$.  For each $k \in \mathbb{N}$ and $a \in A$, recall that $\sigma_{k}(\pi(a)\lvert D_{\nu_{\phi}} \rvert^{-1})$ denotes the $k$-th largest singular value of the operator $\pi(a)\lvert D_{\nu_{\phi}} \rvert^{-1} \in B(\operatorname{ker}(D_{\nu_{\phi}})^{\perp})$.  Further, since $D_{\nu_{\phi}}$ has a compact resolvent, $\sigma_{k}(\pi(a) \lvert D_{\nu_{\phi}} \rvert^{-1})$ converges to zero as $k$ tends to infinity.  Therefore, fixing $x \in \Sigma_{\mathpzc{A}}^{*} \cup \{\emptyset\}$, for each $k \in \mathbb{N}$, there exists $\eta_{{k}} \in \mathbb{N}_{0}$ such that
\[
r^{\eta_{{k}}} \; \leq \; \sigma_{k}(\sigma_{k}(\chi_{{[x]}}) \lvert D_{\nu_{\phi}} \rvert^{-1}) \; < \;r^{\eta_{{k}} - 1}
\]
and such that $\eta_{{k}}$ tends to infinity as $k$ tends to infinity.  Hence, there exists a positive constant $c$ such that for each sufficiently large $N \in \mathbb{N}$, we have that
\begin{equation}\label{Equation-1}
\Xi_{x}( M r^{\eta_{{N}} - 1} ) \, \leq \, \sum_{k = 1}^{N} \sigma_{k}(\pi(\chi_{{[x]}}) \lvert D_{\nu_{\phi}} \rvert^{-1}) \, \leq \, c + \Xi_{x} (r^{\eta_{{N}}})
\end{equation}
and
\begin{equation}\label{Equation-2}
\ln(\Upsilon_{x}( M r^{\eta_{{N}} - 1})) \, \leq \, \ln(N) \, \leq \, \ln(c + M\Upsilon_{x}(r^{\eta_{{N}}})).
\end{equation}
The pressure of the potential function $\phi - P(\phi, \sigma)$ is equal to zero and by Theorem \ref{PFR-E-M-Uniq-thm}, we have that the unique equilibrium measure $\nu_{\phi - P(\phi, \sigma)}$ for the potential function $\phi - P(\phi, \sigma)$ is equal to $\nu_{\phi}$.  Using this, the inequalities given in Equations (\ref{Equation-1}) and (\ref{Equation-2}) and by the results of Corollary \ref{Cor1} and Corollary \ref{Cor2}, we have that
\begin{eqnarray*}
\liminf_{N \to \infty} \frac{\sum_{k = 1}^{N} \sigma_{k}(\pi(\chi_{{[x]}}) \lvert D_{\nu_{\phi}} \rvert^{-1})}{\ln(N)} &\leq& \liminf_{N \to \infty} \frac{c + \Xi_{x} (r^{\eta_{{N}}})}{\ln(\Upsilon_{x}(M r^{\eta_{{N}} - 1}))}\\
&=& \liminf_{N \to \infty} \frac{\ln(r^{\eta_{{N}}}) \, \nu_{\phi}([x])}{\ln(M^{-1} r^{-\eta_{{N}} + 1}) h_{\nu_{\phi}}(\sigma)} \; = \; \frac{\nu_{\phi}([x])}{h_{\nu_{\phi}}(\sigma)}.
\end{eqnarray*}
By a similar argument, one can deduce that
\[
\limsup_{N \to \infty} \frac{\sum_{k = 1}^{N} \sigma_{k}(\pi(\chi_{{[x]}}) \lvert D_{\nu_{\phi}} \rvert^{-1})}{\ln(N)} \, \geq \, \frac{\nu_{\phi}([x])}{h_{\nu_{\phi}}(\sigma)}.
\]
Therefore, for each $x \in \Sigma_{\mathpzc{A}}^{*} \cup \{ \emptyset \}$ we have that the Dixmier trace of the operator $\pi(\chi_{{[x]}}) \lvert D_{\nu_{\phi}} \rvert^{-1}$ is independent of the state $\mathpzc{W}$.  Namely, we have that
\[
\Nint \pi(\chi_{{[x]}}) \lvert D_{\nu_{\phi}} \rvert^{-1} \, = \, \frac{\nu_{\phi}([x])}{h_{\nu_{\phi}}(\sigma)}.
\]
Subsequently, by Definition \ref{metricdimdefnst}, it follows that the metric dimension of $(A, H, D)$ is equal to $1$.  Now, for each state $\mathpzc{W}$ (satisfying the conditions of \cite[Theorem 1.5]{Rennie2}), the operator defined, for each $a \in A$, by
\[
a \mapsto \text{Tr}_{\mathpzc{W}}(\pi(a) \lvert D_{\nu_{\phi}} \rvert^{-1}),
\]
is a bounded linear functional on $A$.  Hence, by the Riesz Representation Theorem there exists a finite Borel measure $\nu$ such that, for each $a \in A$, we have that
\[
\text{Tr}_{\mathpzc{W}}(\pi(a) \lvert D_{\nu_{\phi}} \rvert^{-1}) \, = \, \int_{\Sigma_{\mathpzc{A}}} a \, d\nu.  
\]
Further, the set $R \mathrel{:=} \{ [x] \, : \, x \in \Sigma_{\mathpzc{A}}^{*}\} \cup \{ \emptyset, \Sigma_{\mathpzc{A}} \}$ forms a semi-ring on which the set function $\Lambda: R \to [0, \infty)$ given, for each $I \in R$, by
\[
\Lambda(I) \, \mathrel{:=} \, \text{Tr}_{\mathpzc{W}}(\pi(\chi_{{I}}) \lvert D_{\nu_{\phi}} \rvert^{-1}),
\]
is an additive $\sigma$-additive set function.  Therefore, since $\Lambda$ is also $\sigma$-finite, by the Hahn-Kolmogorov Theorem, for an arbitrary state $\mathpzc{W}$ (satisfying the conditions of \cite[Theorem 1.5]{Rennie2}) and for each $a \in A$, we have that
\[
\text{Tr}_{\mathpzc{W}}(\pi(a) \lvert D_{\nu_{\phi}} \rvert^{-1}) \, = \, \frac{1}{h_{\nu_{\phi}}(\sigma)} \int_{\Sigma_{\mathpzc{A}}} a \, d \nu_{\phi}.
\]
Namely, for each $a \in A$, we have that
\[
\Nint \pi(a) \lvert D_{\nu_{\phi}} \rvert^{-1} \, = \, \frac{1}{h_{\nu_{\phi}}(\sigma)} \, \int_{\Sigma_{\mathpzc{A}}} a \; d \nu_{\phi}.
\]
\end{proof}

\begin{rmk}
In \cite{Bellisard}, Bellissard and Pearson presents an alternative spectral triple to that considered here, which represents the full shift space $(\Sigma^{\infty}, \sigma)$ on two symbols equipped with an ultra-metric.  An example of such an ultra metric is given, for $\omega, \upsilon \in \Sigma^{\infty}$ by
\[
d_{\nu_{\phi}}(\omega, \upsilon) \; \mathrel{:=} \; \inf \{ \nu_{\phi}([x]) \; : \; x \in \Sigma^{*} \cup \{ \emptyset \} \; \text{and} \; \omega, \upsilon \in [x] \},
\]
where $\nu_{\phi}$ is an equilibrium measure for a H\"older continuous potential function $\phi \in C(\Sigma^{\infty}; \mathbb{R})$.  For such a metric, our results give that the noncommutative volume constant of Bellissard and Pearson's spectral triple is equal to $2/h_{\nu_{\phi}}(\sigma)$.   Another recent construction of a spectral triple, which is also interesting in this context, is presented by Sharp in \cite{Sharp-R}.
\end{rmk}

\section*{Acknowledgement}

The authors are grateful to the referee for carefully reading the original article and for many valuable suggestions.  The second author was mainly supported by grants, EP/P50273X/1 and EP/PHDPLUS/AMC3/DTG2010 and in part by ARC - Noncommutative Fractal Geometry: New Invariants.

\begin{small}

\end{small}
\end{document}